\documentclass[11pt]{amsart}

\usepackage[T1]{fontenc}
\usepackage{times}
\usepackage{amssymb,epsfig,verbatim}
\usepackage{calligra,mathrsfs}
\theoremstyle{plain}
\usepackage[T1]{fontenc}
\usepackage[utf8]{inputenc}
\usepackage[english]{babel}
\usepackage{amsthm}
\usepackage{graphics}
\usepackage{amsmath}
\usepackage{amstext}
\usepackage{multirow}
\usepackage{enumerate}
\usepackage{enumitem}
\usepackage{hyperref}
\usepackage{graphicx}
\usepackage{amsmath}
\usepackage{amssymb}
\usepackage{amsthm}
\usepackage{amstext}
\usepackage{epstopdf}
\usepackage{mathrsfs}
\usepackage{amscd}
\usepackage{tikz}
\usepackage{tikz-cd}
\AtBeginDocument{%
   \def\MR#1{}
}

\usepackage{float}
\usetikzlibrary{matrix}

\usepackage{todonotes}

\newtheorem{thm}{Theorem}[section]
\newtheorem{assumption}{Assumption}[section]

\newtheorem{dfn}[thm]{Definition}
\newtheorem{lemma}[thm]{Lemma}
\newtheorem{prop}[thm]{Proposition}

\newtheorem{cor}[thm]{Corollary}

\newtheorem{conj}[thm]{Conjecture}

\newtheorem{ex}[thm]{Example}

\theoremstyle{remark}
\newtheorem{remark}[thm]{Remark}

\newcommand{\mb}{\mathbb}

\newcommand{\mc}{\mathcal}

\newcommand{\C}{\mb C}

\newcommand{\F}{\mc F}
\newcommand{\G}{\mc G}

\newcommand\restr[2]{{
  \left.\kern-\nulldelimiterspace 
  #1 
  \vphantom{\big|} 
  \right|_{#2} 
  }}
\tikzset{
    table/.style={
       matrix of nodes,
        row sep=-\pgflinewidth,
        column sep=-\pgflinewidth,
        nodes={
            rectangle,
            draw=black,
            align=center
        },
        minimum height=1.5em,
        text depth=0.5ex,
        text height=2ex,
        nodes in empty cells,
        every even row/.style={
            nodes={fill=gray!20}
        },
        column 1/.style={
            nodes={text width=3.0cm,align=center,font=\small}
        },
        column 2/.style={
            nodes={text width=3.4cm,align=center},font=\small},
        column 3/.style={
            nodes={text width=5.0cm,align=center},font=\small},
        row 1/.style={
            nodes={ font=\bfseries
            }
        }
   }
}
\newcommand{\Fol}[3]{\textsf{Fol}_{#1}(\mathbb P_{#3}^{#2})}

\newcommand{\Lin}[3]{\textsf{Lin}_{#1}(\mathbb P_{#3}^{#2})}
\newcommand{\Log}[3]{\textsf{Log}_{(#1)}(\mathbb P_{#3}^{#2})}
\newcommand{\SLog}[3]{\textsf{SLog}_{(#1)}(\mathbb P_{#3}^{#2})}

\DeclareMathOperator{\coker}{coker}

\DeclareMathOperator{\codim}{codim}
\DeclareMathOperator{\sing}{sing}

\DeclareMathOperator{\Pic}{Pic}

\DeclareMathOperator{\Div}{Div}

\DeclareMathOperator{\Proj}{Proj}

\DeclareMathOperator{\Spec}{Spec}
\DeclareMathOperator{\Spm}{Spm}

\DeclareMathOperator{\ord}{ord}
\DeclareMathOperator{\Hom}{Hom}
\DeclareMathOperator{\Frob}{Frob}

\DeclareMathOperator{\Ram}{Ram}
\DeclareMathOperator{\Branch}{Branch}

\DeclareMathOperator{\HomSheaf}{\mathscr{H}\text{\kern -3pt {\calligra\large om}}\,}
\DeclareMathOperator{\CohSheaf}{\mathscr{H}}

\newcommand{\NF}{{N_\F}}

\newcommand{\BC}{{\nabla^\F}}
\newcommand{\CC}{{\nabla^{\mathrm{can}}}}

\newcommand{\ie}{{\it{i.e., }}}

\newcommand{\Vertical}{\mathcal V(\varphi)}

\newcommand{\distmin}[2]{{\delta_{#1}(#2)}}
\newcommand{\correction}[2]{{\mathrm{defect}({#1},{#2})}}

\newcommand{\BBconnection}{{\nabla^{\mathrm{B}}}}
\newcommand{\pdivisor}[1]{{\Delta_{#1}}}
\newcommand{\pcurvature}[1]{{\psi_{#1}}}
\newcommand{\pdegeneracy}[1]{{\Delta_{#1}}}
\newcommand{\pkernel}[1]{{\mathscr V({#1})}}
\newcommand{\pclosure}[1]{{\overline{#1}^{p}}}
\newcommand{\CartierTransform}[1]{{\mathscr{C}(#1)}}
\newcommand{\Cartier}{\mathbf{C}}
\newcommand{\field}{{\mathsf k}}

\newcommand{\different}[2]{\mathrm{diff}_{#2}({#1})}
\newcommand{\Witt}[2]{W_{#1}{(#2)}}

\DeclareMathOperator{\End}{End}
\DeclareMathOperator{\identity}{id}
\DeclareMathOperator{\Quot}{Quot}

\numberwithin{equation}{section}
\addtocounter{section}{0}             
\numberwithin{equation}{section}       

\title[Foliations in positive characteristic]{Codimension one foliations in positive characteristic}
\author{Wodson Mendson}
\author{Jorge Vit\'orio Pereira}

\date{\today}

\begin{document}

\begin{abstract}
    We investigate the geometry of codimension one foliations on smooth projective varieties defined over fields
    of positive characteristic with an eye toward applications to the structure of codimension one holomorphic foliations on projective
    manifolds.
\end{abstract}

\maketitle
\setcounter{tocdepth}{1}
\tableofcontents

\section{Introduction}
The study of foliations (subsheaves of the tangent sheaf closed under Lie brackets) on algebraic varieties defined over fields of characteristic $p>0$ traditionally focuses on foliations that are closed under $p$-th powers. For instance, Miyaoka studied these objects in order to establish the existence of rational curves on complex projective varieties in the presence of subsheaves of the tangent sheaf with positivity properties, see \cite{MR927960}, \cite[Chapter 9]{MR1225842}, and \cite[Part I, Lecture III]{miyaoka1997geometry}.

In the first part of this work, we focus instead on codimension one foliations which are not closed under $p$-th powers. We exploit the 20-years-old observation \cite{MR1904485}, \cite[Theorem 6.2]{MR2324555} that these foliations admit non-trivial infinitesimal transverse symmetries and, as such, can be defined by closed rational $1$-forms, uniquely defined up to multiplication by constants of derivations (\ie $p$-th powers of rational functions). The action of the Cartier operator on these $1$-forms defines a companion distribution (not necessarily involutive subsheaf of the tangent sheaf),  the Cartier transform of the foliation. The intersection of this distribution with the original foliation is a canonically defined subdistribution that we begin to study in this work.

Although we believe that the study of the  geometry of foliations in positive characteristics is worth pursuing per se, our primary motivation is the potential for applications toward the study of holomorphic foliations in the spirit of earlier joint work of Loray, Touzet and the second author, \cite[Section 7]{MR3842065}.  In the second part of this work, we use the results established in the first part to give new information about the space of codimension one foliations on projective spaces; see, for instance, Theorem \ref{T:P1P1}, which exhibits previously unknown irreducible components,  and Theorem \ref{T:characterize log components}, which characterizes the so-called logarithmic components.

\subsection{Acknowledgements}
Mendson acknowledges support from CNPq and Faperj. Pereira acknowledges support from CNPq (Grant number 301683/2019-0), and FAPERJ (Grant number E-26/202.530/2019). We express our sincere appreciation to Frédéric Touzet for generously permitting us to incorporate Proposition \ref{P:closed}, Example \ref{Ex:Touzet}, and Proposition \ref{P:relativelyclosed} in this paper, which are all attributed to him. Additionally, we extend our thanks to João Pedro dos Santos for reviewing an initial version of this work and offering valuable suggestions to enhance the clarity of the exposition. We would also like to express our gratitude to an anonymous referee for their meticulous review, which helped identify and rectify several inaccuracies.

\section{Distributions and foliations}
In this section, we present the definitions of distributions and foliations on arbitrary smooth varieties defined over an algebraically closed field $\field$.

\subsection{Distributions}
A distribution $\mathcal D$ on a smooth algebraic variety consists of a pair $(T_{\mathcal D}, N^*_{\mathcal D})$  of coherent subsheaves
of $T_X$ and $\Omega^1_X$ such that $T_\mathcal D$ is the annihilator of $N^*_{\mathcal D}$ and $N^*_{\mathcal D}$ is the annihilator of $T_{\mathcal D}$. Explicitly,
\begin{align*}
    T_{\mathcal D} &= \{ v \in T_X \, | \, \omega(v) = 0 \text{ for every } \omega \in N^*_{\mathcal D} \} \, \text{ and}\\
    N^*_{\mathcal D} &= \{ \omega \in \Omega^1_X \, | \, \omega(v)=0 \text{ for every } v \in T_{\mathcal D} \} \, .
\end{align*}
The sheaf $T_{\mathcal D}$ is the tangent sheaf of the distribution, and the sheaf $N^*_{\mathcal D}$ is the conormal sheaf of the distribution. It follows from the definition that both sheaves $T_{\mathcal D}$ and $N^*_{\mathcal D}$ are reflexive, \ie canonically isomorphic to their bi-duals. The dimension of $\mathcal D$ is the generic rank of $T_{\mathcal D}$ and the codimension of $\mathcal D$ is the generic rank of $N^*_{\mathcal D}$.

The dual of $T_{\mathcal D}$ is the cotangent sheaf of $\mathcal D$ and will be denoted by $\Omega^1_{\mathcal D}$. Similarly, the dual of $N^*_{\mathcal D}$ is the normal sheaf of $\mathcal D$ and will be denoted by $N_{\mathcal D}$.  These sheaves fit into the following exact sequences.
\begin{center}
    \begin{tikzcd}[row sep=tiny]
        0  \arrow{r} & T_{\mathcal D} \arrow{r} & T_{X} \arrow{r} & N_{\mathcal D} \\
        0  \arrow{r} & N^*_{\mathcal D} \arrow{r} & \Omega^1_X \arrow{r} & \Omega^1_{\mathcal D}
    \end{tikzcd}
\end{center}

For every $i$ between $1$  and $\dim \mathcal D$, we will define $\Omega^i_{\mathcal D}$ as the dual of $\wedge^i T_{\mathcal D}$, i.e.,
\[
    \Omega^i_{\mathcal D}  = \Hom\left( \bigwedge^i T_{\mathcal D}, \mathcal O_X \right) \, .
\]
Since the sheaves $\Omega^i_{\mathcal D}$ are defined as duals, they are reflexive sheaves. It is convenient to set $\Omega^0_{\mathcal D} = \mathcal O_{\mathcal D}$  and $\Omega^j_{\mathcal D} = 0$ if $j \notin \{0, \ldots \dim \mathcal D\}$.

We will denote the determinant of $\Omega^1_{\mathcal D}$ by $\omega_{\mathcal D}$, i.e.,
\[
    \omega_{\mathcal D} = \left(\bigwedge^{\dim {\mathcal D}} \Omega^1_{\mathcal D}\right)^{**} \, .
\]

Since $T_{\mathcal D}$ is reflexive, \cite[Proposition 1.10]{hartshorne1980stable} implies that
\[
    \Omega^1_{\mathcal D}   \simeq    \left(\bigwedge^{\dim \mathcal D - 1} T_{\mathcal D}\right)^{**}\otimes \omega_{\mathcal D}  \, .
\]

The singular set of $\mathcal D$ is the locus where $T_X/T_{\mathcal D}$ is not locally free and is denoted by $\sing(\mathcal D)$. Since $T_X/T_{\mathcal D}$ is torsion-free, the codimension of $\sing(\mathcal D)$ is at least two. Outside $\sing(\mathcal D)$, the rightmost arrows of both exact sequences above are surjective, see for instance \cite[Section 2]{GiraldoPan10}. Taking determinants one gets an isomorphism of line-bundles
\begin{equation}\label{E:adjunction for distributions}
   \omega_{X} \simeq \det(\Omega^1_{\mathcal D}) \otimes \det(N^*_{\mathcal D}) \, ,
\end{equation}
where $\omega_X$ denotes the canonical sheaf of $X$.

\begin{remark}\label{R:doubledual}
    Any coherent torsion-free sheaf injects into its double dual, see for instance  \cite[\href{https://stacks.math.columbia.edu/tag/0AVT}{Tag 0AVT}, Lemma 31.12.4]{stacks-project}.  Therefore  we can see $T_X/T_{\mathcal D}$ as a subsheaf of $N_{\mathcal D}$.
\end{remark}

If $\mathcal D$ is a distribution of codimension $q$ then there exists
$
    \omega \in H^0(X, \Omega^q_X \otimes \det(N_{\mathcal D})) \, ,
$
without codimension one zeros, such that $T_{\mathcal D}$ is the kernel of the morphism
\[
    T_X \to \Omega^{q-1}_X \otimes \det(N_{\mathcal D})
\]
defined by contraction with $\omega$.

\subsection{Foliations}
A foliation $\F$ on a smooth algebraic variety $X$ is a distribution with tangent sheaf $T_{\F}$ closed under Lie brackets. Explicitly, the $\mathcal O_X$-morphism  (also called the O'Neill tensor)
\begin{align*}
    \bigwedge^2 T_{\mathcal F} & \longrightarrow N_{\F} \\
    v\wedge w &\mapsto [v,w] \mod T_{\mathcal F}
\end{align*}
is identically zero. Here we used Remark \ref{R:doubledual} in order to replace $T_X/T_{\F}$ by $N_{\F}$. Coherent subsheaves of $T_X$ which are closed under Lie brackets are called involutive subsheaves.

\begin{lemma}\label{L:Frobenius Integrability}
    Let $\F$ be a foliation of dimension $r$ on a smooth algebraic variety $X$. If $x \notin \sing(\F)$ then there exist $v_1, \ldots, v_r$ generators of $T_{\F} \otimes \mathcal O_{X,x}$ such that $[v_i, v_j]=0$ for every $i,j \in \{ 1, \ldots, r\}$.
\end{lemma}
\begin{proof}
    See \cite[Lemma 6.1]{MR2324555}.
\end{proof}

If $X$ is a complex algebraic variety then the vector fields given by Lemma \ref{L:Frobenius Integrability} can be integrated, in an analytic or formal neighborhood of $x$, to a germ of action of an abelian Lie group of dimension $r$. The situation in positive characteristic is utterly different as will be discussed at length in Section \ref{S:charp}.

A coherent subsheaf $\mathcal L$ of $\Omega^1_X$ is called integrable if
$d \mathcal L$ is contained in the saturation of $\mathcal L \wedge \Omega^1_X$ in $\Omega^2_X$.
Cartan's formula for the exterior derivative (see \cite[Lemma 3]{MR0106223} for its validity in positive characteristic)  implies that the involutiveness of $T_{\mathcal F}$ is equivalent to the integrability of $N^*_{\mathcal F}$.

\begin{ex}
    The foliation $\F$ on $\mathbb A^2$ with conormal sheaf generated by $\omega = xdy - \lambda y dx$, $\lambda$ different from $0$ and $-1$,
    is such that $d N^*_{\F}$ is not contained in $N^*_{\mathcal F} \wedge \Omega^1_X$.
\end{ex}
\begin{proof}
    Notice that $d \omega = (1+ \lambda) dx \wedge dy$. If $\lambda \neq 0, 1$ then $\omega$ vanishes only at $0$ ($\lambda \neq 0$) and
    $d\omega$ vanishes nowhere ($\lambda \neq -1$). Therefore it does not exist a regular $1$-form $\eta$ such that $d\omega = \eta \wedge \omega$.
\end{proof}

The integrability of $N^*_{\mathcal F}$ and the reflexiveness of $\Omega^{\bullet}_{\F}$ implies that the exterior derivative of differential forms on $X$ descends to a $\field$-linear morphism
\[
    d_{\F} : \Omega^{\bullet}_{\mathcal F} \to \Omega^{\bullet+1}_{\mathcal F}
\]
which, like the exterior derivative, satisfies Leibniz' rule. Indeed, if $U \subset X$ is a sufficiently small  open subset
contained in $X- \sing(\F)$  and $\omega$ belongs to $\Omega^i_{\F}(U)$ then we consider any lift $\hat \omega$ of $\omega$ to $\Omega^i_{X}(U)$  and define $d_{\F} \omega$ as the restriction of $d \hat \omega$ to $\wedge^{i+1} T_{\F}$. If $\hat \omega'$ is any other lift, then $\hat \omega - \hat \omega'$ belongs to $\Omega^{i-1}_X(U) \wedge N^*_{\F}(U)$ and the integrability condition implies that $d ( \hat \omega - \hat \omega' )$ vanishes on $\wedge^{i+1} T_{\F}$. If $j : X-\sing(\F) \to X$ denotes the inclusion, then we have just shown how to  define a $\field$-linear morphism $d_{\F}$ on $j^*\Omega^{\bullet}_{\F}$. The reflexiveness of $\Omega^{\bullet}_{\F}$ implies that $d_{\F}$ extends to a $\field$-linear morphism defined on the whole $X$ which inherits the Leibniz' property from $d$.

Every smooth variety $X$ has two trivial foliations. One foliation of dimension zero given by the pair $(0,\Omega^1_X) \subset (T_X,\Omega^1_X)$, called the foliation by points, and one foliation of dimension $\dim X$ given by the pair $(T_X,0)$, called the foliation with just one leaf.

\part{Geometry of foliations in positive characteristic}
\section{Differential geometry in positive characteristic}\label{S:diff geo charp}

In this section, for later use, we briefly review the most fundamental properties of vector fields in positive characteristic (Subsection \ref{SS:vector fields}), and the Cartier correspondence (Subsection \ref{SS:Cartier Correspondence}) for coherent sheaves endowed with flat connections with vanishing $p$-curvature.

\subsection{Absolute Frobenius}
Let $X$ be a scheme over an algebraically closed field $\field$ of characteristic $p>0$. The absolute Frobenius morphism of $X$ is the endomorphism  $\Frob = \Frob_X : X \to X$ of the scheme $X$, which acts as the identity on the topological space $X$ and which acts on the structural sheaf of $X$ by raising elements to their $p$-th powers.

\begin{dfn}
    Let $\mathcal E_1$ and $\mathcal E_2$ be coherent sheaves of $\mathcal O_X$-modules. A morphism
    $\varphi : \mathcal E_1 \to \mathcal E_2$ of sheaves of abelian groups is called
    \begin{enumerate}
        \item $p$-linear if $\varphi(f e) = f^p e$; and
        \item $p^{-1}$-linear if $\varphi(f^p e) = f e$
    \end{enumerate}
    for every $f \in \mathcal O_X(U)$, $e \in \mathcal E_1(U)$ and every open subset $U \subset X$.
\end{dfn}

A $p^{-1}$-linear morphism $\varphi: \mathcal E_1 \to \mathcal E_2$ is equivalent to an $\mathcal O_X$-linear morphism $\varphi: \Frob_* \mathcal E_1 \to \mathcal E_2$, see \cite[Sections 2 and 3]{BlickleSchwede13}. If $\mathcal E$ is a coherent $\mathcal O_X$-module then  $\mathcal E$ and $\Frob_* \mathcal E$  are isomorphic, but the $\mathcal O_X$-module structure on $\Frob_* \mathcal E$ is different. If we denote the action of $\mathcal O_X$ on $\Frob_* \mathcal E$ by $\star$ then
\[
    f \star \sigma = f^p \sigma
\]
for any $f \in \mathcal O_X$ and any $\sigma \in \Frob_* \mathcal E$.

If $\mathcal E$ still denotes a coherent sheaf of $\mathcal O_X$-modules then
\begin{align*}
    \Frob^* \mathcal E & = \mathcal O_X \otimes_{\Frob^{-1} \mathcal O_X} \Frob^{-1} \mathcal E && \text{(Definition of pullback)} \\
    & = \mathcal O_X  \otimes_{\Frob^{-1} \mathcal O_X} \mathcal E && \text{($\Frob$ is the identity on $|X|$).}
\end{align*}
Consequently, in $\Frob^* \mathcal E$ we have that
\begin{equation}\label{E:Frob^*}
    1 \otimes f \sigma = f^p \otimes \sigma
\end{equation}
for any $f \in \mathcal O_X$ and any $\sigma \in \mathcal E$. It follows that a $p$-linear morphism $\varphi: \mathcal E_1 \to \mathcal E_2$ is equivalent to an $\mathcal O_X$-linear morphism $\varphi : \Frob^* \mathcal E_1 \to \mathcal E_2$.

The $\mathcal O_X$-module $\Frob^* \mathcal E$ comes endowed with a canonical connection
\begin{align*}
    \CC : \Frob^* \mathcal E & \to \Omega^1_X \otimes \Frob^* \mathcal E \\
    f \otimes \sigma & \mapsto df \otimes (1 \otimes \sigma) \, .
\end{align*}
It follows from Equation (\ref{E:Frob^*}) that the kernel of $\CC$ can be identified, as sheaves of abelian groups, with $1 \otimes \mathcal E \subset \Frob^* \mathcal E$.

\subsection{Vector fields and the $p$-curvature of connections}\label{SS:vector fields}
Given a vector field/derivation $v \in T_X(U)$ on an open subset $U \subset X$ of a variety defined over a field $\field$ of characteristic $p>0$, then the $p$-th iteration of $v$ is also a derivation since Leibniz rule implies
\[
    v^p(f \cdot g) = \sum_{i=0}^p \binom{p}{i} v^{i}(f) \cdot v^{p-i}(g) = f v^{p}(g) + v^p(f) g
\]
for every $f,g \in \mathcal O_X(U)$.

Thus, one can define a morphism of sheaves of sets $\times^{p} : T_X \to T_X$ which takes a vector field and raises it to its $p$-th power.
This is not a morphism of sheaves of abelian groups, but a result of Jacobson (\cite[Equation (5.2.4)]{Katz1970}) says that \begin{equation}\label{E:Jacobson}
    (v + w)^p = v^p + w^p + Q_p(v,w)
\end{equation}
where $Q_p$ is a Lie polynomial, i.e. a polynomial on (iterated Lie) brackets of $v$ and $w$.  Moreover, if $g \in \mathcal O_X$ and
$v \in T_X$ then
\begin{equation}\label{E:Deligne}
    (f v)^p = f^p v^p - f v^{p-1}(f^{p-1}) v \, ,
\end{equation}
according to \cite[Equation (5.4.0)]{Katz1970} modulo a sign misprint.

Recall from \cite[Section 5]{Katz1970} that the $p$-curvature of a connection $\nabla : \mathcal E \to \Omega^1_X \otimes \mathcal E$ is the
morphism of $\mathcal O_X$-modules
\begin{align*}
    \psi : \Frob^* T_X & \longrightarrow \End_{\mathcal O_X}(\mathcal E) \\
    v & \mapsto (\nabla_v)^p - \nabla_{v^p} \, ,
\end{align*}
where $\nabla_\bullet$ is the  $\field$-endomorphisms of $\mathcal E$ that send defined by
\begin{align*}
    \mathcal E & \longrightarrow \mathcal E \\
    \sigma & \mapsto (i_\bullet \otimes \identity) \nabla\sigma \, .
\end{align*}
where $i_{\bullet}$ denotes the contraction of $1$-forms with the vector field $\bullet$.

\begin{ex}
    The canonical connection $\CC$ defined on $\Frob^* \mathcal E$ is a flat connection with  zero $p$-curvature.
\end{ex}

\subsection{Relative Frobenius, $p$-closed flat connections, and the Cartier correspondence}\label{SS:Cartier Correspondence}
The absolute Frobenius is not a morphism of $\field$-schemes. Nevertheless, if we set $S= \Spec \field$ and
consider the fiber product  $X^{(p)} = X \times_{S} S$ of the structure morphism of $X$ with the absolute Frobenius of $S$ then the absolute
Frobenius $\Frob$ factors as in the following commutative diagram
\begin{center}
    \begin{tikzcd}
        X \arrow[rr,"\Frob_{X/S}"]  \arrow[rrrr,bend left,"\Frob=\Frob_X"] \arrow{drr} && X^{(p)} \arrow{d} \arrow[rr,"W_{X/S}"] && X \arrow{d} \\
        && S \arrow[rr,"\Frob_S"] && S \, . \\
    \end{tikzcd}
\end{center}
The morphism $\Frob_{X/S} : X \to X^{(p)}$ is a morphism of $\field$-schemes and is called the relative Frobenius morphism of $X$ over $S$, see  \cite[Éxpose XV]{SGA5}. Notice that $W_{X/S} : X^{(p)} \to X$ is an isomorphism of schemes (but not of $\field$-schemes) defined through
the twisting of the coefficients by means of the field automorphism $x \mapsto x^p$.

Given any coherent sheaf $\mathcal E$ of $\mathcal O_X$-modules on $X$, we set $\mathcal E^{(p)} = W_{X/S}^* \mathcal E$.
In other words, $\mathcal E^{(p)}$ is the
base change of $\mathcal E$ to $X^{(p)}$. As such, $\mathcal E^{(p)}$ is an $\mathcal O_{X^{(p)}}$-module
such that $\Frob^* \mathcal E = \Frob_{X/S}^* \mathcal E^{(p)}$. If we consider $\mathcal O_{X^{(p)}}$ as a subsheaf of $\mathcal O_X$
then the kernel of the canonical connection $\CC$ on $\Frob_{X/S}^* \mathcal E^{(p)}$ is an $\mathcal O_{X^{(p)}}$-module
that can be identified with $\mathcal E^{(p)}$. Reciprocally, if $\mathcal E$ is a coherent sheaf on $X$ endowed with a flat connection
$\nabla$ of $p$-curvature zero then the kernel of $\nabla$ is, in a natural way, an $\mathcal O_{X^{(p)}}$-module such that
\[
    \Frob_{X/S}^* (\ker \nabla) \simeq \mathcal E \, .
\]
This is essentially the content of the Cartier correspondence: there exists an equivalence of categories between the category of
quasi-coherent sheaves on $X^{(p)}$ and the category of quasi-coherent $\mathcal O_X$-modules with flat connections of $p$-curvature zero
given by
\[
    \mathcal E \mapsto (\Frob_{X/S}^* \mathcal E , \CC) \quad \text{ and } \quad (\mathcal E, \nabla) \mapsto \ker \nabla,
\]
see, for instance, \cite[Theorem 5.1]{Katz1970}.

\begin{lemma}\label{L:lift}
    Let $\mathcal E$ be a coherent sheaf of $\mathcal O_X$-modules. Let $\mathcal K \subset \Frob^*\mathcal E$
    be a coherent subsheaf.
    If the morphism of $\mathcal O_X$-modules
    \begin{align*}
        \varphi: \mathcal K & \longrightarrow \Omega^1_X \otimes \frac{\Frob^* \mathcal E}{\mathcal K}
    \end{align*}
    induced by $\CC$ is identically zero then there exists a coherent subsheaf $\mathcal L \subset \mathcal E$ such that $\mathcal K = \Frob^* \mathcal L$
\end{lemma}
\begin{proof}
    If $\varphi$ vanishes then the sheaf $\mathcal K$ is preserved by the canonical connection of $\Frob^* \mathcal E$. By Cartier correspondence
    there exists a subsheaf $\mathcal L^{(p)}\subset \mathcal E^{(p)}$ on $X^{(p)}$ such that $\mathcal K = \Frob_{X/S}^* \mathcal L^{(p)}$. Since $W_{X/S}$ is an isomorphism, and $\mathcal E^{(p)}= W_{X/S}^* \mathcal E$,  there exists a subsheaf $\mathcal L \subset \mathcal E$ on $X$ such that $\mathcal L^{(p)} = W_{X/S}^* \mathcal L$. The sheaf $\mathcal L$ is such that $\mathcal K = \Frob^* \mathcal L$.
\end{proof}

\section{Particular features of foliations in positive characteristic}\label{S:charp}
In this section, we introduce the key concepts for the first part of the paper: the $p$-curvature of the foliation (Subsection \ref{SS:p-curvature}), $p$-dense foliations (Subsection \ref{SS:pdense}), the degeneracy divisor of the $p$-curvature, the kernel of the $p$-curvature of a foliation (Subsection \ref{SS:pkernel}), and  the Cartier transform of a foliation (Subsection \ref{SS:Cartier transform}).

\subsection{The $p$-curvature of foliations}\label{SS:p-curvature}
Let $\F$ be a foliation on $X$. It follows from Equations (\ref{E:Jacobson}) and (\ref{E:Deligne}) that the map which sends $v \in T_{\F}$ to the class of $v^p$  in ${T_X}/{T_{\F}} \subset N_{\F}$ is $p$-linear and additive. We can thus define the morphism of $\mathcal O_X$-modules
\begin{align*}
    \pcurvature{\F} : \Frob^* T_{\F} & \longrightarrow N_{\F} \\
    \sum f_i \otimes v_i & \mapsto \sum f_i v_i^p \mod T_{\F} \, .
\end{align*}
The morphism $\pcurvature{\F}$ is called the $p$-curvature of $\F$. If $\pcurvature{\F} =0$ then we say that the
foliation is $p$-closed.

The foliation $\F$ determines a subsheaf $\mathcal O_{X/\F} \subset \mathcal O_X$
consisting of functions that are annihilated by any section of $T_{\F}$, that is
\[
    \mathcal O_{X/\F} = \{ a \in \mathcal O_X \, ; \, v(a) =0 \text{ for any } v \in T_\F \} \, .\
\]

\begin{thm}\label{T:pintegra}
    If $\F$ is a $p$-closed foliation of dimension $r$ on a smooth algebraic variety $X$
    then $Y = \Spec_{\mathcal O_{X^{(p)}}} \mathcal O_{X/\F}$ is a normal algebraic variety
    and the inclusion $\mathcal O_{X/\F} \subset \mathcal O_X$ determines a morphism $f:X \to Y$
    which factors the  relative Frobenius $\Frob_{X/\Spec \field} : X \to X^{(p)}$ and satisfies
    $[\field(X): f^* \field(Y)]=p^r$.
    Reciprocally, if $\mathcal O\subset \mathcal O_X$ is a sheaf of $\mathcal O_{X^{(p)}}$-algebras   integrally closed in
    $\mathcal O_X$ then there exists a $p$-closed foliation $\F$ such that $\mathcal O= \mathcal O_{X/\F}$.
\end{thm}
\begin{proof}
    See  \cite[Proposition 1.9, Lecture II, Part I]{miyaoka1997geometry} or \cite[Theorem 2]{Seshadri1960}.
\end{proof}

\begin{cor}
    Let $\F$ be a foliation of codimension $q$ on a smooth algebraic variety $X$. The foliation $\F$ is $p$-closed
    if, and only if, there exists $q$ rational first integrals for $\F$, say $f_1, \ldots, f_q \in \field(X)$, such that
    $df_1 \wedge \cdots \wedge df_q \neq 0$.
\end{cor}
\begin{proof}
    See \cite[Section 7]{MR3842065}.
\end{proof}

\subsection{$p$-dense foliations}\label{SS:pdense}
Some authors use the term foliation only to refer to $p$-closed foliations. In this work
we will also consider foliations which are not $p$-closed.

\begin{dfn}
    Let $\F$ be a foliation on a smooth variety $X$. The $p$-closure of $\F$  is the (unique) $p$-closed
    foliation $\pclosure{\F}$ such that $\mathcal O_{X/\F} = \mathcal O_{X/\pclosure{\F}}$. When the $p$-closure
    of $\F$ coincides with the foliation with only one leaf, \ie, $\mathcal O_{X/\F} = \mathcal O_{X^{(p)}}$,
    we will say that $\F$ is $p$-dense.
\end{dfn}

Notice that a foliation is $p$-dense if, and only if, any rational first integral $f \in \field(X)$ of $\F$ is
such that $df =0$.

\begin{prop}\label{P:old}
    If $\F$ is a $p$-dense codimension one foliation on a smooth algebraic variety $X$ defined by $\omega \in H^0(X,\Omega^1_X \otimes N_{\F})$
    and $v$ is a rational vector field tangent to $\F$ such that $\omega(v^p) \neq 0$ then the rational $1$-form
    \[
        \frac{\omega}{\omega(v^p)}
    \]
    is closed.
\end{prop}

The result above first appeared in  \cite[Theorem 6.2]{MR1394970}. Its proof can be generalized to $p$-dense foliations of arbitrary codimension as shown by the next proposition.

\begin{prop}\label{P:closed}
    If $\F$ is a $p$-dense foliation of codimension $q$ on a smooth algebraic variety $X$ then there exists
    closed rational $1$-forms $\omega_1, \ldots, \omega_q$ which vanish on $T_{\F}$ and satisfy $\omega_1 \wedge \cdots \wedge \omega_q \neq 0$.
\end{prop}
\begin{proof}
    Let $r = \dim \F$. If $x \notin \sing(\F)$ then there exists generators $v_1, \ldots, v_r$ of $T_{\F} \otimes \mathcal O_{X,x}$ such
    that $[v_i , v_j] =0$ for every $i, j \in \{ 1, \ldots, r\}$, see Lemma \ref{L:Frobenius Integrability}. Notice that the iterated $p$-th powers of the vector fields $v_i$ commute
    with $v_1, \ldots, v_r$ and also among themselves. Since $\F$ is $p$-dense, we can choose among the iterated $p$-powers of $v_1, \ldots, v_r$, vector fields  $v_{r+1}, \ldots, v_{r+q}$, $r+q = \dim X$, such that
    \[
        v_1 \wedge \cdots \wedge v_r \wedge v_{r+1} \wedge \cdots \wedge v_{r+q}
    \]
    does not vanish identically. Therefore, we can interpret $v_1, \ldots, v_{r+q}$ as $\field(X)$-linearly independent commuting rational vector fields on $X$.

    Let $\alpha_1, \ldots, \alpha_{r+q}$ be rational $1$-forms on $X$ dual to $v_1, \ldots, v_{r+q}$. In other words, the $1$-forms $\alpha_i$ are characterized by $\alpha_i(v_j)= \delta_{ij}$. We claim that the  $1$-forms $\alpha_i$ are closed. As $v_1, \ldots, v_{r+q}$ are $\field(X)$-linearly independent rational vector fields, it suffices to
    check that $d\alpha_i(v_j,v_k) =0 $ for $i,j,k \in \{ 1, \ldots, r+q\}$. Since $v_1, \ldots, v_{r+q}$ commute, Cartan's formula for the exterior derivative implies that
    \[
        d \alpha_i(v_j,v_k) = v_j(\alpha_i(v_k)) - v_k(\alpha_i(v_j)) - \alpha_i([v_j,v_k]) = 0
    \]
    establishing our claim.  By construction, $\omega_1 = \alpha_{r+1}, \ldots, \omega_q = \alpha_{r+q}$ vanish on $T_{\F}$ and satisfy $\omega_1 \wedge \cdots \wedge \omega_q  \neq0$ as wanted.
\end{proof}

The following example shows that without  $p$-denseness of $\F$, Proposition \ref{P:closed}  does not hold.

\begin{ex}\label{Ex:Touzet}
    If the $\field$ has characteristic two, $X = \mathbb A^3_k$ and $\F$ is the foliation on $X$ generated by $v =  x \frac{\partial}{\partial x} + zy \frac{\partial}{\partial y}$
    then every closed rational $1$-form containing $T_{\F}$ in its kernel is a rational multiple of $dz$.
\end{ex}

A simple adaptation of the arguments used in the proof of Proposition \ref{P:closed} gives the following generalization.

\begin{prop}\label{P:relativelyclosed}
    If $\F$ is a not necessarily $p$-dense foliation on a smooth algebraic variety $X$, $\G= \pclosure{\F}$, and $q = \codim (\F) - \codim (\G)$  then there exists rational sections $\omega_1, \ldots, \omega_q$ of $\Omega^1_{\G}$ which vanish on $T_{\F}$ such that
    $\omega_1 \wedge \cdots \wedge \omega_q \neq 0$ and $d_{\G} \omega_1 = \cdots = d_{\G} \omega_q = 0$.
\end{prop}

A $\F$-partial connection on a coherent sheaf $\mathcal E$ is a $\field$-linear morphism $\nabla: \mathcal E \to \Omega^1_{\F} \otimes \mathcal E$ which, like usual connections,
satisfy Leibniz's rule $\nabla(f s) = f \nabla(s) + d_{\F} f \otimes s$. Bott's partial connection is the $\F$-partial connection on the normal sheaf of $\F$ which sends
the class of $v \mod T_{\F}$ to the $1$-form (defined only along $T_{\F}$) that maps a local section $w$ of $T_{\F}$ to $[v,w] \mod T_{\F}$.

Our next result provides an interpretation of Propositions \ref{P:closed} and \ref{P:relativelyclosed} in terms of Bott's partial connection and
the canonical connection on $\Frob^* T_{\F}$.

\begin{prop}\label{P:pcurvature is a F-morphism}
    The morphism $\pcurvature{\F}$ is such that
    \[
        (r \otimes \pcurvature{\F}) \circ \CC = \BC \circ \pcurvature{\F}
    \]
    where $r : \Omega^1_{X} \to \Omega^1_{\F}$ is the natural restriction morphism, $\CC$ is the canonical connection
    on $\Frob^*T_{\F}$, and  $\BC$ is Bott's connection on $N_{\F}$.
    In other words,  $\pcurvature{\F}$ is a morphism of $\F$-partial connections
    between $(\Frob^* T_{\F}, \restr{\CC}{T_{\F}})$ and $(N_{\F}, \BC)$.
\end{prop}
\begin{proof}
    It suffices to show that for every local section $v \in T_{\F}$, the image of $1 \otimes v \in \Frob^*T_{\F}$ under $\pcurvature{\F}$ is a flat
    section of  Bott's partial connection on $N_{\F}$. To do this we proceed as in the proof of Proposition \ref{P:closed}. We fix $x \in X$
    sufficiently general and use Lemma \ref{L:Frobenius Integrability} to guarantee the existence of commuting generators $v_1, \ldots, v_r$ of
    $T_{\F} \otimes \mathcal O_{X,x}$. Since $\pcurvature{\F}$ is $\mathcal O_X$-linear, it suffices to show that
    $v_1^p, \ldots, v_r^p$ are flat sections of Bott's partial connection. But this is clear since each of $v_i^p$ commutes with generators
    of $T_{\F}\otimes \mathcal O_{X,x}$.
\end{proof}

\subsection{The degeneracy divisor and the kernel of the $p$-curvature}\label{SS:pkernel}
Proposition \ref{P:pcurvature is a F-morphism} implies that the kernel of $\pcurvature{\F}$ is invariant
by the partial connection $\restr{\CC}{T_{\F}}$. Therefore the morphism
\[
      \ker \pcurvature{\F}  \longrightarrow \Omega^1_\F \otimes \frac{\Frob^* T_{\F}}{\ker \pcurvature{\F}}
\]
induced by  $\restr{\CC}{T_{\F}}$ is zero. This is not sufficient to apply Lemma \ref{L:lift} as it may happen that $\ker \pcurvature{\F}$
is not invariant by  the  connection $\CC$, that is the morphism of $\mathcal O_X$-modules
\begin{equation}\label{E:obstruction}
      \ker \pcurvature{\F}  \longrightarrow \Omega^1_X \otimes \frac{\Frob^* T_{\F}}{\ker \pcurvature{\F}}
\end{equation}
induced by $\CC$ is not necessarily zero.

\begin{lemma}\label{L:kernel pcurvature}
    If $\F$ is $p$-dense then the morphism (\ref{E:obstruction}) is zero. In particular, there exists a distribution $\pkernel{\F}$ contained
    in $\F$ such that
    \[
        \ker \pcurvature{\F} = \Frob^* T_{\pkernel{\F}} \, .
    \]
    In this case, the distribution $\pkernel{\F}$ will be called the kernel of the $p$-curvature of $\mathcal F$.
\end{lemma}
\begin{proof}
    The vanishing of the morphism (\ref{E:obstruction}) is assured by \cite[Proposition 6.1]{MR3167130}. The existence of the distribution $\pkernel{\F}$ follows from Lemma \ref{L:lift}.
\end{proof}

\begin{dfn}
    If $\F$ is a $p$-dense codimension one foliation then the image of the $p$-curvature morphism $\pcurvature{\F}$
    is of the form $N_{\F} \otimes \mathcal I_Z$ for some ideal sheaf $\mathcal I_Z$. The divisorial part of $\mathcal I_Z$ will be called
    the degeneracy divisor of the $p$-curvature and will be denoted by $\pdegeneracy{\F}$.
\end{dfn}

\begin{ex}\label{E:loggeneric}
    Let $\F$ be a codimension one foliation on $X= \mathbb A^n_{\field}$ with $N^*_{\F}$ generated  by  $1$-form
    \[
        \omega = \left( \prod_{i=1}^r x_i \right) \left( \sum_{i=1}^r \lambda_i \frac{dx_i}{x_i} \right) \, ,
    \]
    where $1<r\le n$ and $\lambda_1, \ldots, \lambda_r \in \field^*$. The tangent sheaf of $\F$ is locally free and  generated
    by the vector fields $v_{j} = \lambda_j x_1 \frac{\partial }{\partial x_1}  - \lambda_i x_j \frac{\partial }{\partial x_j}$ for $1<  j\le r$ and
    $\frac{\partial }{\partial x_{r+1}}, \ldots, \frac{\partial }{\partial x_n}$. Observe that $\frac{\partial^p }{\partial x_{r+1}}=\ldots= \frac{\partial^p }{\partial x_n}=0$ and that
    \[
        \omega(v_{j}^p) = \left( \prod_{i=1}^r x_i \right) ( \lambda_j^p \lambda_1 - \lambda_1^p \lambda_j ) \, ,
    \]
    for any $1 < j \le r$.
    Hence,  $\F$ is $p$-closed if, and only if, $(\lambda_1: \ldots : \lambda_r) \in \mathbb P^{r-1}_{\field}(\mathbb F_p)$. Furthermore,  if $\F$ is not $p$-closed then $\Delta_{\F}$ is the simple normal crossing divisor given
    by $\{ x_1 \cdots x_r =0\}$,
\end{ex}

\begin{prop}\label{P:degree degeneracy}
    Let $\F$ be a $p$-dense codimension one foliation on a smooth algebraic variety $X$ of characteristic $p>0$.
    If $\pkernel{\F}$ denotes the kernel of the $p$-curvature of $\F$ then the identity
    \[
        \mathcal O_X( \pdegeneracy{\F}) = (\omega_{\F}  \otimes \omega_{\pkernel{\F}}^{*})^{\otimes p} \otimes \det N_{\F} \,
    \]
holds in the Picard group of $X$.
\end{prop}
\begin{proof}
    The definition of $\pdegeneracy{\F}$ implies that the sequence
    \begin{equation}\label{E:basic sequence in char p}
 	  0 \longrightarrow \Frob^{*}T_{\pkernel{\F}} \longrightarrow \Frob^{*}T_{\mathcal{F}}
        \longrightarrow N_{\mathcal{F}}(-\Delta_{\mathcal{F}})\longrightarrow 0	
    \end{equation}
    is exact on the complement of a codimension two subset of $X$. The result follows by taking
    determinants.
\end{proof}

\begin{prop} \label{P:hinvariant}
    Let $X$ be a smooth algebraic variety  of characteristic $p>0$ and
    let $\mathcal{F}$ be a $p$-dense codimension one foliation on $X$. Then every $\F$-invariant hypersurface is contained in the
    support of $\pdegeneracy{\F}$. Reciprocally, if $p$ does not divide the coefficient of an irreducible component $H$ in
    $\Delta_{\mathcal F}$  then $H$ is $\mathcal{F}$-invariant.
\end{prop}
\begin{proof}
    Suppose first that $H$ is $\mathcal{F}$-invariant. Let $U \subset X$ be a sufficiently small open subset intersecting $H$
    and let $h \in \mathcal O_X(U)$ be a reduced equation for $H\cap U$. Since $H$ is $\F$-invariant, for any local section
    $v \in T_{\F}(U)$ we have that $v(h)$ is a multiple of $h$. Therefore, the same holds true for $v^p(h)$. In other words,
    for any $v \in T_{\F}(U)$, the restriction of $v$ and $v^p$ to $H\cap U$ is contained in the kernel of $\restr{dh}{H\cap U}$.
    Hence, the support of $\pdegeneracy{\F}$ contains $H$.

    Suppose now that $\ord_{H}(\Delta_{\mathcal{F}}) = \alpha \not\equiv 0 \mod p$. In the notation above, after restricting $U$ if
    necessary, we can chose commuting generators $v_1, \ldots, v_r$ of $T_{\F}(U)$ such that
    \[
        v_1\wedge \cdots \wedge v_r \wedge v_1^p  = h^{\alpha} \theta
    \]
    where $r = \dim \F$, and  $\theta \in \bigwedge^{r+1} T_X(U)$ does not vanish along $H$, see Lemma \ref{L:Frobenius Integrability}. Observe that
    \[
        [v_1, h^{\alpha} \theta] = 0 \implies \alpha v_1(h) h^{\alpha -1} \theta = - h^{\alpha} [v_1, \theta] \, .
    \]
    Since $\alpha \not\equiv 0 \mod p$, it follows that $h$ divides $v_1(h)$. This is sufficient to show that
    $H = \{ h=0\}$ is $\F$-invariant.
\end{proof}

\begin{prop} \label{P:closed + divisor}
    Let $X$ be a smooth algebraic variety  of characteristic $p>0$ and
    let $\mathcal{F}$ be a $p$-dense codimension one foliation on $X$.
    If $\eta$ is a closed rational $1$-form defining $\F$ then the coefficients of
    the divisors $\pdegeneracy{\F}$ and  $((\eta)_{\infty} - (\eta)_{0})$  coincide modulo $p$.
\end{prop}
\begin{proof}
    Let $\eta'$ be another closed rational $1$-form defining $\F$. As both $\eta$ and $\eta'$ define the same codimension one foliation,
    there exists $h \in \field(X)$ such that $\eta = h \eta'$. The closedness of $\eta$ and $\eta'$ implies that $dh \wedge \eta' =0$.
    Since $\F$ is $p$-dense, $dh =0$ and we can write
    $\eta = g^p \eta'$ for some $g \in \field(X)$.  In particular, the coefficients of the divisors $(\eta)_{\infty} - (\eta)_0$
    and $(\eta')_{\infty} - (\eta')_0$ agree modulo $p$.

    Let $U$ be a sufficiently small open subset of $X - ( \sing(\F) \cup \sing(\pkernel{\F}) )$.  Choose a nowhere vanishing section $\omega$ of
    $N^*_{\F}(U)$ and a vector field $v \in T_{\F}(U) \subset T_X(U)$ everywhere transverse to $\restr{T_{\pkernel F}}{U}$.
    Therefore, $f = \omega(v^p)$ generates $\mathcal O_{U}(-\pdegeneracy{\F})$ and the rational $1$-form $\omega/f$ is closed according
    to Proposition \ref{P:old}. Thus  the coefficients of $(\eta)_{\infty} - (\eta)_0$  and
    $\pdegeneracy{\F}$ along hypersurfaces intersecting $U$ agree modulo $p$. Since we can choose $U$ intersecting any given hypersurface of $X$, the result follows.
\end{proof}

\subsection{The Cartier operator}
Let $X$ be a smooth algebraic variety defined over an algebraically closed field of characteristic $p>0$.

The de Rham complex $(\Omega^{\bullet}_X,d)$ of $X$ is a complex of sheaves of abelian groups but is a complex of $\mathcal O_X$-modules. Nevertheless, for any $f \in \mathcal O_X$ and any $\omega\in \Omega^i_X$ we have that
\[
    d (f^p \cdot \omega) = f^p \cdot \omega \, .
\]
It follows that $(\Frob_* \Omega^{\bullet}_X, d)$ is a complex of $\mathcal O_X$-modules. If we denote by $Z\Omega^{\bullet}_X$ the sheaf of closed differential, and by $B \Omega^{\bullet}_X$ the sheaf of locally exact differentials then both $\Frob_* Z \Omega^{\bullet}_X$ and $\Frob_* B \Omega^{\bullet}_X$ are
coherent $\mathcal O_X$-modules. The sheaf cohomology groups of the complex $(\Frob_* \Omega^{\bullet}_X, d)$
are the $\mathcal O_X$-modules
\[
    \CohSheaf^i \Frob_* \Omega^{\bullet}_X =  \frac{\Frob_*Z \Omega^i_X }{\Frob_*B \Omega^i_X } \, .
\]

Cartier proved in \cite{MR84497} the existence of a unique morphism  $\Cartier:\Frob_*Z \Omega^{\bullet}_X  \to \Omega^{\bullet}_X$
of sheaves of graded-commutative $\mathcal O_X$-algebras such that
\begin{enumerate}
    \item $\ker \Cartier = \Frob_* B \Omega^{\bullet}_X$; and
    \item $\Cartier(f^{p-1} df) = df$ for every $f \in \mathcal O_X$; and
    \item $\Cartier$ is surjective and, consequently, induces an isomorphism of graded-commutative $\mathcal O_X$-algebras
    between $\CohSheaf^{\bullet} \Frob_* \Omega^{\bullet}_X$ and $\Omega^{\bullet}_X$.
\end{enumerate}
The morphism $\Cartier$ is the (absolute) Cartier operator on $X$. If $\omega$ is a closed $q$-form, we will
say that $\Cartier(\omega)$ is the Cartier transform of $\omega$

\begin{remark}
    In the literature, one  also finds the relative Cartier operator $\Frob_*Z \Omega^{\bullet}_X  \to \Omega^{\bullet}_{X^{(p)}}$. As we are working over $S=\Spec(\field)$ where $\field$ is an  algebraically closed field, the two
    operators are essentially equivalent, cf. \cite[pages 122-123]{Illusie96}. For details on the construction of the Cartier operator, see for instance  \cite[Section 7]{Katz1970}, \cite[Chapter 1, Section 1.3]{BrionKumar05}, or \cite[Section 3]{Illusie96}.
\end{remark}

\begin{lemma}\label{L:Cartier formula}
    Let $X$ be a smooth variety defined over a field $\field$ of characteristic $p>0$.
    If $\omega \in \Omega^1_X$ and $v \in T_X$ then
    \[
        ( i_v \Cartier(\omega) ) ^p = i_{v^p} \omega - v^{p-1}(i_v \omega) \, .
    \]
    More generally, if $\omega \in \Omega^j_X$ and $v \in T_X$ then
    \[
        i_v \Cartier(\omega) = \Cartier\left( i_{v^p} \omega - (L_v)^{p-1}(i_v \omega) \right) \, .
    \]
\end{lemma}
\begin{proof}
    The first formula is due to Cartier and appears (without proof) in \cite{MR84497}. A proof can be found in \cite[Proposition 3]{Seshadri1960}. The second formula is the content of \cite[Proposition 2.6]{Kaledin08}.
\end{proof}

\subsection{The Cartier transform of a $p$-dense foliation}\label{SS:Cartier transform}
Let $\F$ be a $p$-dense foliation of codimension $q$ on a smooth algebraic variety $X$. Let $\omega_1, \ldots, \omega_q$ be closed rational $1$-forms defining $\F$. The existence of such $1$-forms is guaranteed by Proposition \ref{P:closed}. If $\omega$ is any closed $1$-form
vanishing on $T_{\F}$, we can write
\[
    \omega = \sum_{i=1}^q f_i \omega_i
\]
for some rational functions $f_i \in \field(X)$. If we differentiate the expression above and take the wedge  product of the result with $\omega_1 \wedge \cdots \wedge \widehat{\omega_i} \wedge \cdots \wedge \omega_q$, we see that $df_i \wedge \omega_1 \wedge \cdots \wedge \omega_q=0$ for any $i$. Since $\F$ is $p$-dense, it follows that $df_i =0$ for every $i$, i.e., there exists $g_i \in \field(X)$ such that $f_i = g_i^p$. Therefore the identity
\[
    \Cartier(\omega ) = \sum_{i=1}^q g_i \Cartier(\omega_i) \,
\]
holds. It follows that the distribution defined by the Cartier transform of all closed rational $1$-forms vanishing on $T_{\F}$ can be defined by the Cartier transform of only $q = \codim \F$ closed rational $1$-forms.

\begin{dfn}
    Let $\F$ be a $p$-dense foliation of codimension $q$ on a smooth algebraic variety $X$. The Cartier transform of $\F$ is the distribution $\CartierTransform{\F}$ on $X$ defined by the intersection of the kernels of the Cartier transform of all closed rational $1$-forms     vanishing on $T_{\F}$.
\end{dfn}

\begin{ex}
    The Cartier transform of a $p$-dense foliation is not necessarily a foliation. For instance, the codimension one foliation on $\mathbb A^3_{\field}$ defined by the closed $1$-form $\omega = y^{p-1}dy + z^p x^{p-1} dx$ has Cartier transform defined by the non-integrable $1$-form
    $\Cartier(\omega) = dy + z dx$.
\end{ex}

\subsection{The Cartier transform and the kernel of the $p$-curvature}
The proposition below gives an alternative  description of the kernel of the $p$-curvature of a $p$-dense foliation $\F$.

\begin{prop}
    Let $\F$ be a $p$-dense foliation on a smooth algebraic variety $X$. The kernel of the $p$-curvature of $\F$ coincides with the saturation
    of $T_{\F} \cap T_{\CartierTransform{\F}}$ inside $T_X$.
\end{prop}
\begin{proof}
    Lemma \ref{L:kernel pcurvature} implies the existence of a distribution $\pkernel{\F}$ on $X$ such that $\ker \pcurvature{\F} = \Frob^* T_{\pkernel{\F}}$.

    Let $v \in T_{\F}$. For any closed rational $1$-form vanishing on $T_{\F}$, Lemma \ref{L:Cartier formula} implies that
    \[
        ( i_v \Cartier(\omega) ) ^p = i_{v^p} \omega - v^{p-1}(i_v \omega) = i_{v^p} \omega \, .
    \]
    If $v \in T_{\pkernel{\F}}$ then, by the definition of $\pkernel{\F}$, $i_{v^p} \omega =0$ for any $\omega \in N^*_{\F}$. Therefore, when $\omega$ is closed, $i_v \Cartier(\omega) =0$. This shows that $T_{\pkernel{\F}}$ is contained in the saturation of $T_{\F} \cap T_{\CartierTransform{\F}}$.

    Reciprocally, if $v \in T_{\F} \cap T_{\CartierTransform{\F}}$ then, by definition,  $i_{v} \Cartier(\omega)=0$ for any closed rational rational $1$-form $\omega$ vanishing on $T_{\F}$. Consequently, $i_{v^p} \omega =0$ for any closed rational section of $N^*_{\F}$. Since $N^*_{\F}$ is generically generated by closed rational $1$-forms, as proved in Proposition \ref{P:closed}, this is sufficient to show that $T_{\F} \cap T_{\CartierTransform{\F}}$ is contained in $T_{\pkernel{\F}}$.
\end{proof}

\subsection{Global expression for the kernel of the $p$-curvature}\label{SS:global expression}
Assume that $\F$ is a $p$-dense foliation of codimension one on a smooth algebraic variety $X$. For later use, we will present a twisted section of $\Omega^1_{\F}$ defining the distribution $\pkernel{\F}$.

Let $U $ be the open set $X - ( \sing(\F) \cup \sing(\pkernel{\F}) )$ and consider a sufficiently fine open covering $U_i$ of $U$. Choose for each $U_i$  a nowhere vanishing section $\omega_i$ of $N^*_{\F}(U_i)$ and a vector field $v_i \in T_{\F}(U_i) \subset T_X(U_i)$ everywhere transverse to $\restr{T_{\pkernel F}}{U_i}$. Therefore, $f_i = \omega_i(v_i^p)$ generates $\mathcal O_{U_i}(-\pdegeneracy{\F})$ and the rational $1$-form $\omega_i/f_i$ is closed.

Let $\eta_i$ be the restriction of the rational $1$-form $\Cartier(\omega_i/f_i)$ to $\restr{T_{\F}}{U_i}$. We claim that $\eta_i$ is a regular section of $\restr{\Omega^1_{\F}}{U_i}$ without codimension one zeros. Indeed, if $v \in T_{\F}(U_i)$ then, using Lemma \ref{L:Cartier formula}, we can write
\[
    \omega_i(v^p) = f_i \left( \frac{\omega_i}{f_i} \right) (v^p)    = f_i \left(\Cartier\left(\frac{\omega_i}{f_i} \right) (v) \right)^p  \, .
\]
Since $\omega_i(v^p) \in \mathcal O_X(-\pdegeneracy{\F})(U_i) = f_i \mathcal O_X(U_i)$, we deduce that $\eta_i (v) = \Cartier\left(\frac{\omega_i}{f_i} \right)(v) \in \mathcal O_X(U_i)$ for no matter which $v \in T_{\F}(U_i)$. Moreover, $\eta_i$ has no zeros since
$\eta_i(v_i) =1$.

Taking into account Proposition \ref{P:degree degeneracy}, we deduce that the collection $\{ \eta_i \}$ defines a regular section $\eta$ of the sheaf $\restr{(\Omega^1_{\F} \otimes {\omega_{\pkernel \F}} \otimes \omega^*_{\F})}{U} = \restr{(\Omega^1_{\F} \otimes \det{N_{\pkernel \F}} \otimes N^*_{\F})}{U}$. Since $\Omega^1_{\F} \otimes {\omega_{\pkernel \F}} \otimes \omega^*_{\F}$ is reflexive, $\eta$ extends to an element of
\[
    H^0(X, \Omega^1_{\F} \otimes {\omega_{\pkernel \F}} \otimes \omega^*_{\F}) = H^0(X, \Omega^1_{\F} \otimes \det{N_{\pkernel \F}} \otimes N^*_{\F}) \, ,
\]
without codimension one zeros, defining the distribution $\pkernel{\F} \subset \F$.

\section{Degeneracy divisor of the $p$-curvature}
This section investigates the degeneracy divisor of the $p$-curvature of codimension one foliations. It starts by describing how the degeneracy divisor behaves under two natural operations: pull-back by dominant rational maps (Subsection \ref{SS:pull-back of degeneracy divisor}) and restriction to non-invariant subvarieties (Subsection \ref{SS:restriction}). It then moves, in Subsection \ref{SS:p-divisor}, to review work by the first author on the reducibility of the degeneracy divisor of the $p$-curvature of sufficiently general foliations on the projective plane and Hirzebruch surfaces.  The results on the reducibility and on the behavior under pull-backs will be essential for the applications developed in the second part of the paper.

\subsection{Equidimensional and dominant rational maps}\label{SS:pb00}
Let $X$ and $Y$ be smooth varieties of dimensions $n$ and $m$, respectively, defined over an algebraically closed field $\field$  and let $\varphi :X \dashrightarrow Y$ be a dominant rational map. Throughout this section we will make use of the following assumption on $\varphi$.

\begin{assumption}\label{A:rational}
    There exist  open subsets $X^{\circ} \subset X$ and $Y^{\circ} \subset Y$ such that $\varphi$ restricts to a morphism
    $\varphi^{\circ} : X^{\circ} \to Y^{\circ}$
    satisfying
    \begin{enumerate}
        \item\label{IA:cod} the set $X- X^{\circ}$ has codimension at least two;
        \item\label{IA:equi} $\varphi^{\circ}$ is equidimensional and, consequently,  the Zariski closure in $Y^{\circ}$ of the image of any irreducible hypersurface of $X^{\circ}$ is either a hypersurface of $Y^{\circ}$ or the whole $Y^{\circ}$;
        \item\label{IA:diff} the differential of $\varphi^{\circ}$ has generically rank equal to $m=\dim Y$ and the differential of the restriction of $\varphi^{\circ}$ to any irreducible hypersurface $H \subset X^{\circ}$ has generic rank equal to $m=\dim Y^{\circ}$ when $H$ dominates $Y^{\circ}$ or equal to $m-1$ when the image of $H$ is contained in a hypersurface.
    \end{enumerate}
\end{assumption}

When $\field$ has characterisitic zero, Item (\ref{IA:diff}) follows from
Item (\ref{IA:equi}), but this is no longer true when $\field$ has positive characteristic.

Define the ramification  divisor $\Ram(\varphi^{\circ}) \in \Div(X^{\circ})$  of $\varphi^{\circ}$ as
\begin{equation}\label{E: rel canonical for varphi}
    \Ram(\varphi^{\circ}) = \sum_{H} (\varphi^{\circ})^* H - ((\varphi^{\circ})^* H)_{\mathrm{red} } \, ^,
\end{equation}
where the sum runs over all irreducible hypersurfaces of $Y^{\circ}$. Assumption \ref{A:rational}, Item (\ref{IA:diff}), guarantees that
this is a finite sum.  Moreover, each irreducible component of the support of $\Ram(\varphi^{\circ})$ is mapped (dominantly) to an
irreducible hypersurface of $Y^{\circ}$. Define the branch divisor of $\varphi^{\circ}$, $\Branch({\varphi^{\circ}})$, as the reduced
divisor with support equal to the image of $\Ram(\varphi^{\circ})$. Set $\Ram(\varphi)$ as the divisor on $X$ given by the closure of  $\Ram(\varphi^{\circ})$.

If $\mathcal L$ is a line-bundle on $Y$ then we will abuse the notation and write $\varphi^* \mathcal L$ the unique line-bundle on
$X$ such that $\restr{\varphi^* \mathcal L}{X^{\circ}} = (\varphi^{\circ})^*\left(\restr{\mathcal L}{Y^\circ} \right)$. Since $X$ is smooth and $\codim X- X^{\circ} \ge 2$, such line-bundle always exists.

The next proposition, as well as Proposition \ref{P:pb dominant} below, are well-known for equidimensional morphisms between normal complex varieties, see for instance \cite[Subsection 2.9]{MR3695844}.

\begin{prop}\label{P:deg vs rel can}
    Let $\varphi :X \dashrightarrow Y$ be a rational map satisfying Assumption \ref{A:rational}. If $\Vertical$ is the foliation
    on $X$ defined by $\varphi$ then
    \[
        \omega_{\Vertical} = \omega_X \otimes \varphi^* \omega_Y^* \otimes \mathcal O_X(- \Ram(\varphi)) \, .
    \]
\end{prop}
\begin{proof}
    Further restricting $X^{\circ}$ and $Y^{\circ}$ (keeping the assumption $\codim X- X^{\circ} \ge 2$), we can, and will,  assume that all
    irreducible components of $\Ram(\varphi^\circ)$ and $\Branch(\varphi^\circ)$   are smooth, pairwise disjoint and that the rank
    of the differential of $\varphi^{\circ}$ is equal to $m-1=\dim Y-1$ on  the support of $\Ram(\varphi^{\circ})$ and equal to
    $m$ everywhere else.

    Set $B= \Branch(\varphi^{\circ})$ and $R = ((\varphi^{\circ})^*(B))_{red}$.
    Notice that $\Ram(\varphi^{\circ}) = (\varphi^{\circ})^*B - R$. As usual, we will write $T_{X^\circ}(-\log R)$ for
    the subsheaf of $T_{X^\circ}$ formed by vector fields tangent to $R$ and similarly for $T_{Y^\circ}(-\log B)$.

    Let $\Vertical$ be the foliation defined by $\varphi$. Assumption \ref{A:rational}, Item (\ref{IA:diff}), (or rather the stronger version of it made on the first paragraph of this proof)  implies that the tangent sheaf of
    $\restr{\Vertical}{X^{\circ}}$ fits into the exact sequence
    \[
        0 \longrightarrow  \restr{T_{\Vertical}}{X^0} \longrightarrow T_{X^{\circ}}(-\log R)
        \longrightarrow (\varphi^{\circ})^* (T_{Y^{\circ}}(- \log B)) \to 0 \,
    \]
    where the morphism $T_{X^{\circ}}(-\log R) \longrightarrow (\varphi^{\circ})^* (T_{Y^{\circ}}(- \log B))$  is induced the differential
    $d\varphi^{\circ}$ of $\varphi^{\circ}$. It suffices to take the determinant of this exact sequence to conclude.
\end{proof}

\subsection{Pull-backs of foliations under dominant rational maps}\label{SS:pb0}

Before continuing the discussion we introduce two notations:
\begin{enumerate}
    \item If $D$ is an arbitrary divisor on smooth variety carrying a foliation $\F$, we will write $D = D_{\F} + D_{\F^{\perp}}$ for
    the unique decomposition of $D$ as sum of divisors where all the irreducible components of the support of $D_{\F}$ are $\F$-invariant and
    all irreducible components of the support of $D_{\F^{\perp}}$ are not $\F$-invariant.
    \item We will write $T_{\F}(-\log D)$ for the subsheaf of $T_{\F}$ formed by vector fields tangent to $\F$ and $D$. This definition implies that     $T_{\F}(-\log D)= T_{\F}(-\log D_{\F^{\perp}})$ and $T_{\F}(-\log D_{\F})= T_{\F}$. Note that $\det T_{\F}(-\log D) = \det T_{\F} \otimes \mathcal O_X(-D_{\F^{\perp}})$.
\end{enumerate}

\begin{prop}\label{P:pb dominant}
    Let $\varphi :X \dashrightarrow Y$ be a rational map satisfying Assumption \ref{A:rational}, and let $\G$ be a foliation on $Y$.
    If $\F$ is the pull-back of $\G$ under $\varphi$ then
    \begin{eqnarray*}
        \det( N^*_{\F} ) & = & \varphi^* \det(N^*_{\G}) \otimes \mathcal O_{X}\left( \Ram(\varphi)_{\F} \right),  \text{ and } \\
        \omega_{\F}  &= &\varphi^* \omega_{\G} \otimes \omega_{\Vertical} \otimes \mathcal O_{X}\left( \Ram(\varphi)_{\F^{\perp}}\right).
    \end{eqnarray*}
\end{prop}
\begin{proof}
    The restriction of the  tangent sheaf of $\F$  to $X^{\circ}$ is defined by the
    exact sequence
    \[
        0 \longrightarrow \restr{T_{\F}}{X^{\circ}} \longrightarrow T_{X^{\circ}} \longrightarrow \varphi^*\left( \restr{N_{\G}}{Y^\circ}\right)
    \]
    where the rightmost arrow is given by the composition of the differential of $\varphi^{\circ}$ with the natural projection
    $T_{Y^0} \to \restr{N_{\G}}{Y^{\circ}}$.

    As in the proof of Proposition \ref{P:deg vs rel can}, set $B= \Branch(\varphi^{\circ})$ and $R = ((\varphi^{\circ})^*(B))_{red}$. Assumption \ref{A:rational}, Item (\ref{IA:diff}), gives an exact sequence
    \[
        0 \longrightarrow \restr{T_{\Vertical}}{X^0} \longrightarrow \restr{T_{\F}}{X^{\circ}} (-\log R_{\F^{\perp}}) \longrightarrow (\varphi^{\circ})^*\left(\restr{T_{\G}}{Y^{\circ}}(- \log B_{\G^{\perp}})\right)
    \]
    with the rightmost morphism surjective in codimension one. Taking determinants we obtain that
    \[
        \omega_{\F}  = \varphi^* \omega_{\G}  \otimes \omega_{\Vertical} \otimes \mathcal O_{X}\left( \Ram(\varphi)_{\F^{\perp}} \right)\, ,
    \]
    as claimed. The formula for the conormal bundle of $\F$ follows from Equation (\ref{E:adjunction for distributions})
    (adjunction for foliations) combined with Proposition \ref{P:deg vs rel can}.
\end{proof}

\subsection{Degeneracy divisor of the pull-back under a dominant rational map}\label{SS:pull-back of degeneracy divisor}
Throughout this subsection, $\field$ is an algebraically closed field of characteristic $p>0$. We will determine a formula comparing the degeneracy divisor of the $p$-curvature of a $p$-dense codimension one foliation $\G$ on a smooth projective variety with the degeneracy divisor of its pull-back under a dominant rational map  $\varphi :X \dashrightarrow Y$   satisfying Assumption \ref{A:rational}.

\begin{prop}\label{P:pdegeneracy of pb}
    Let $\varphi :X \dashrightarrow Y$ be a rational map, and let $\G$ be a $p$-dense codimension one foliation on $Y$.
    If $\varphi$ satisfies Assumption \ref{A:rational} and $\F = \varphi^* \G$ then $\pdegeneracy{\F}$ is linearly equivalent to
    \[
         \varphi^* \pdegeneracy{\G} - \Ram(\varphi)_{\F} + p \left( \Ram(\varphi)_{{\F^{\perp}}} - \Ram(\varphi)_{{\pkernel{\F}}^{\perp}}  \right)  \, .
    \]
\end{prop}
\begin{proof}
    The proof builds on a combination of Proposition \ref{P:degree degeneracy} with Proposition \ref{P:pb dominant}.

    First observe that $\pkernel{\F}$,  the kernel of $p$-curvature of $\F = \varphi^* \G$, coincides with $\varphi^* \pkernel{\G}$.
    Therefore, Proposition \ref{P:pb dominant} implies that
    \begin{eqnarray*}
        N_{\F} &=& \varphi^* N_{\G} \otimes \mathcal O_X\left( - \Ram(\varphi)_{\F}  \right) \\
        \omega_{\F} &=&  \varphi^* \omega_{\G } \otimes \omega_{\Vertical} \otimes \mathcal O_{X}\left(  \Ram(\varphi)_{\F^{\perp}}   \right)  \\
        \omega_{\pkernel{\F}} &=& \varphi^* \omega_{\pkernel{\G} } \otimes \omega_{\Vertical}
        \otimes \mathcal O_{X}\left( \Ram(\varphi)_{\pkernel{\F}^{\perp}} \right) \, .
    \end{eqnarray*}
    Hence, we deduce from Proposition \ref{P:degree degeneracy} that $\mathcal O_X(\pdegeneracy{\F})$ is isomorphic to
    \[
         \varphi^* \mathcal O_Y(\pdegeneracy{\G}) \otimes
        \mathcal O_X \left(   - \Ram(\varphi)_{\F} + p\left(  \Ram(\varphi)_{{\pkernel{\F}}^{\perp}} - \Ram(\varphi)_{{\F^{\perp}}}\right) \right) \,
    \]
    in the Picard group of $X$.
\end{proof}

\begin{cor}\label{C:pdeg preciso}
    Assumptions as in Proposition \ref{P:pdegeneracy of pb}. Assume also that $X$ has no non-zero effective divisor linearly equivalent to zero.
    If $\Ram(\varphi)_{{\pkernel{\F}}^{\perp}} =\Ram(\varphi)_{{\F^{\perp}}}=0$
    and all the coefficients of $\pdegeneracy{\F}$ are strictly smaller than the characteristic of $\field$ then
    \[
        \pdegeneracy{\F} = \varphi^* \pdegeneracy{\G} - \Ram(\varphi)_{\F} \, .
    \]
\end{cor}
\begin{proof}
    According to Proposition \ref{P:old}, $\G$ is defined by a closed rational $1$-form $\eta$.
    Since  $\varphi^* \eta$ is a closed $1$-form defining $\F$, Proposition \ref{P:closed + divisor} implies that
    \begin{align*}
        \pdegeneracy{\G} &= (\eta)_{\infty} - (\eta)_0 \mod p \, , \\
        \pdegeneracy{\F} &= (\varphi^* \eta)_{\infty} - (\varphi^* \eta)_{0} \mod p  .
    \end{align*}
    A local computation  gives that
    $(\varphi^*\eta)_{\infty} - (\varphi^*\eta)_{0}  = \varphi^*\left( (\eta)_{\infty} - (\eta)_{0} \right) - \Ram(\varphi)_{\F}$
    in accordance with Proposition \ref{P:pb dominant}.
    Consequently, we can write
    \[
        \pdegeneracy{\F} = \varphi^* \pdegeneracy{\G} - \Ram(\varphi)_{\F}  \mod p \, .
    \]
    Our assumption $\Ram(\varphi)_{{\pkernel{\F}}^{\perp}} =\Ram(\varphi)_{{\F^{\perp}}}=0$
    combined with Proposition \ref{P:pdegeneracy of pb} imply that $\pdegeneracy{\F}$ and $\varphi^* \pdegeneracy{\G} - \Ram(\varphi)_{\F}$ are
    linearly equivalent.
    Since $\G$-invariant hypersurfaces are contained $\pdegeneracy{\G}$ according to Proposition \ref{P:hinvariant}, the divisor $\varphi^* \pdegeneracy{\G} - \Ram(\varphi)_{\F}$ is effective. Therefore  $\pdegeneracy{\F}$ and $\varphi^* \pdegeneracy{\G} - \Ram(\varphi)_{\F}$ are effective divisors,  with the same reduction modulo $p$, and the coefficients of $\pdegeneracy{\F}$ are smaller than $p$, it
    follows our assumptions on $X$ that they must be same  divisor.
\end{proof}

\begin{ex}\label{E:behavior ell}
    Let $n\ge 2$ and $\ell >0$ be integers such that $p$ does not divide $\ell$.
    The table below, where $H= \{ x_n=0\}$,  describes the behaviour of the degeneracy divisor of the $p$-curvature of the pull-back $\F = \varphi^* \G$ of a foliation  $\G$ on $\mathbb A^n_{\field}$ under the tame morphism
    \begin{align*}
        \varphi : \mathbb A^n_{\field} \to \mathbb A^n_{\field} \\
        (x_1, \ldots, x_{n-1} ,x_n) & \mapsto (x_1, \ldots, x_{n-1},  x_n^{\ell}) \, .
    \end{align*}
    \begin{tikzpicture}
    \matrix (mat) [table] 
    {
        $H$ is $\F$-invariant & $H$ is $\pkernel{\F}$-invariant     & Degeneracy divisor     \\
        No  & No     & $\varphi^* \pdegeneracy{\G}$       \\
        No  & Yes     & $\varphi^*\pdegeneracy{\G} + p(\ell-1)H$     \\
        Yes & Yes     & $\varphi^* \pdegeneracy{\G} - (\ell-1) H$     \\
    };
    \end{tikzpicture}
\end{ex}

\subsection{Non-invariant hypersurfaces and differents}\label{SS:different}
Let $\F$ be a distribution on a smooth algebraic variety $X$ and let $H \subset X$ be a
smooth hypersurface. Let $\delta$ be  the natural composition
\[
    \restr{T_{\F}}{H} \to \restr{T_X}{H} \to N_H \, .
\]
If $H$ is not $\F$-invariant then the image of $\delta$ is equal to $N_H \otimes \mathcal I$ for
some ideal sheaf $\mathcal I \subset \mathcal O_H$. The effective divisor defined by the codimension one
part of $\mathcal I$ will be called the different of $\F$ and $H$ and will be denoted by $\different{\F}{H}$.
Notice that the kernel of $\delta$ is the tangent sheaf of the distribution $\mathcal H = \restr{\F}{H}$.

Note that when $X$ is a surface and $H$ is a curve, $\different{\F}{H}$ is nothing but the tangency
divisor of $\F$ and $H$ as defined by Brunella \cite[Chapter 3]{MR3328860}. The terminology adopted here is in accordance with \cite[Section 3]{MR4036447} and is justified by the analogy with the homonymous concept from birational geometry (see, for instance, \cite[Section 4.1]{MR3057950}) suggested by the following observation.

\begin{lemma}\label{L:different}
    Let $\F$ be a distribution of arbitrary dimension on a smooth algebraic variety $X$ and let $H \subset X$ be a smooth hypersurface.
    If $H$ is not $\F$-invariant and $\mathcal H= \restr{\F}{H}$ is the restriction of $\F$ to $H$ then
    \begin{eqnarray*}
        \restr{(\omega_{\F} \otimes \mathcal O_X(H))}{H} &\simeq & \omega_{\mathcal H} \otimes \mathcal O_H(\different{\F}{H}) \, , \\
        \restr{(\det N_{\F})}{H} & \simeq & \det N_{\mathcal H} \otimes \mathcal O_H(\different{\F}{H}) \, ,
    \end{eqnarray*}
    where we adopt the convention that when $\mathcal H$ is the foliation/distribution by points then $\omega_{\mathcal H} = \mathcal O_H$ and $N_{\mathcal H} = T_H$.
\end{lemma}
\begin{proof}
    The first formula follows from the fact that the kernel of $\delta$ equals $T_{\mathcal H}$. The second formula is deduced from the first
    by adjunction.
\end{proof}

\begin{lemma}\label{L:differentbis}
    Let $\F$ be a codimension $q$ distribution on a smooth algebraic variety $X$ defined by $\omega \in H^0(X,\Omega^q_X \otimes \det N_{\F})$. Let $\G \subset \F$ be a subdistribution of codimension $r$ in $\F$ defined by $\eta \in H^0(X, \Omega^r_{\F} \otimes \det N_{\G} \otimes \det N^*_{\F})$. If $H\subset X$ is a smooth hypersurface not invariant by $\F$ and $\G$, $i:H \to X$ is the natural inclusion, and $\mathcal H$ is the restriction of $\F$ to $H$ then
    \begin{enumerate}
        \item\label{I:different 1} the zero divisor of $i^* \omega \in H^0(H, \Omega^q_H \otimes \restr{\det N_{\F}}{H})$ is equal to $\different{\F}{H}$; and
        \item\label{I:different 2} the zero divisor of $i^* \eta \in H^0(H, \Omega^r_{\mathcal H} \otimes \restr{(\det N_{\G} \otimes \det N^*_{\F})}{H})$ is equal to
        the difference of differents $\different{\G}{H} - \different{\F}{H}$.
    \end{enumerate}
\end{lemma}
\begin{proof}
    Item (\ref{I:different 1}) is a reinterpretation of the proof of  Lemma \ref{L:different}.
    To verify Item (\ref{I:different 2}), first observe that we can define unambiguously the product $\beta = \omega \wedge \eta \in H^0(X, \Omega^{q+r}_X \otimes \det N_{\G})$ by considering an arbitrary rational lift $\hat \eta$ of $\eta$ to $\Omega^r_X \otimes \det N_{\G} \otimes \det N^*_{\F})$ and wedging it with $\omega$. The resulting twisted $(q+r)$-form does not depend on the lift and is a regular twisted $(q+r)$-form without codimension one zeros defining the distribution $\G$ . According to Item (\ref{I:different 1}), the zero divisor of $i^* \beta$ is equal to $\different{\G}{H}$. It follows that the zero divisor of $i^* \eta$ is equal to the difference $\different{\G}{H} - \different{\F}{H}$ as claimed.
\end{proof}

\subsection{Differents and non-invariant subvarieties of higher codimension}
A variant of the discussion about differents carried out in Subsection \ref{SS:different} can be made for the restriction
of a distribution $\F$  to a smooth subvariety $Y$ of codimension greater than one. One extra difficulty comes from the fact that the natural
morphism
\[
    \delta : \restr{T_{\F}}{Y} \to N_Y
\]
can have any rank between $0$ and $\codim Y$. When it has rank zero, the subvariety $Y$ is $\F$-invariant. If the rank of $\delta$ is equal to $q=\codim Y$ then we will say that $Y$ is strongly transverse to $\F$. In this case, the image of $\wedge^q \delta$ is of the form $\det N_Y \otimes \mathcal I$ for some ideal sheaf $\mathcal I \subset \mathcal O_Y$, and we define the different $\different{\F}{Y}$ as the effective divisor on $Y$ defined by the codimension one part of $\mathcal I$.

The definition implies that
\begin{eqnarray*}
            \restr{(\omega_{\F} \otimes \det N_Y )}{Y} &\simeq & \omega_{\mathcal Y} \otimes \mathcal O_Y(\different{\F}{Y}) \,  \\
            \restr{(\det N_{\F})}{Y} & \simeq & \det N_{\mathcal Y} \otimes \mathcal O_Y(\different{\F}{Y}) \, ,
\end{eqnarray*}
for any smooth subvariety $Y \subset X$ strongly transverse to $\F$. Likewise, we also have a natural analogue of Lemma \ref{L:differentbis}.

\begin{lemma}\label{L:differentbis2}
  Let $\F$ be a codimension $q$ distribution on a smooth algebraic variety $X$ defined by $\omega \in H^0(X,\Omega^q_X \otimes \det N_{\F})$. Let $\G \subset \F$ be a subdistribution of codimension $r$ in $\F$ defined by $\eta \in H^0(X, \Omega^r_{\F} \otimes \det N_{\G} \otimes \det N^*_{\F})$. If $Y\subset X$ is a smooth subvariety strongly transverse to $\F$ and $\G$, $i:Y \to X$ is the natural inclusion, and $\mathcal Y$ is the restriction of $\F$ to $Y$ then
  \begin{enumerate}
        \item\label{I:differentbis 1} the zero divisor of $i^* \omega \in H^0(Y, \Omega^q_Y \otimes \restr{\det N_{\F}}{Y})$ is equal to $\different{\F}{Y}$; and
        \item\label{I:differentbis 2} the zero divisor of $i^* \eta \in H^0(Y, \Omega^r_{\mathcal Y} \otimes \restr{(\det N_{\G} \otimes \det N^*_{\F})}{Y})$ is equal to the difference of differents $\different{\G}{Y} - \different{\F}{Y}$.
  \end{enumerate}
\end{lemma}
\begin{proof}
    Analogous to the proof of Lemma   \ref{L:differentbis}.
\end{proof}

\subsection{The degeneracy divisor of the $p$-curvature for the restriction to strongly transverse subvarieties}\label{SS:restriction}
We are now ready to state and prove a result that offers a comparison between the degeneracy divisor of the $p$-curvature of a foliation and its restriction to a smooth subvariety.

\begin{prop}\label{P:different}
    Let $\F$ be a codimension one foliation on a smooth algebraic variety $X$ defined over a field $\field$ of characteristic $p>0$.
    If $Y\subset X$ is a smooth  subvariety strongly transverse to  by $\pkernel{\F}$ (consequently, strongly transverse to $\F$) and $\mathcal Y$ is the
    restriction of $\F$ to $Y$ then $\pdegeneracy{\mathcal Y}$, the degeneracy divisor of the $p$-curvature of  $\mathcal Y$, is equal to
    \[
         \restr{(\pdegeneracy{\F})}{Y} + p \cdot \big( \different{\pkernel{\F}}{Y} - \different{\F}{Y} \big) - \different{\F}{Y}   \, .
    \]
\end{prop}
\begin{proof}
    Let $\omega \in H^0(X, \Omega^1_X\otimes N_{\F})$ be a twisted $1$-form defining $\F$. According to Lemma \ref{L:differentbis}
    the zero divisor $i^* \omega$ is $\different{\F}{Y}$. We can thus write $i^* \omega = \delta \cdot \alpha$ where $\delta$ belongs to $H^0(Y,\mathcal O_Y(\different{\F}{Y})$ and $\alpha \in H^0(Y,\Omega^1_Y \otimes N_{\mathcal Y})$ is a twisted $1$-form defining $\mathcal Y= \restr{\F}{Y}$.

    If $f \in H^0(X,\mathcal O_X(\pdegeneracy{\F}))$ defines the degeneracy divisor of the $p$-curvature of $\F$ and $\eta \in H^0(X, \Omega^1_{\F} \otimes \det N_{\pkernel{\F}} \otimes N^*_{\F})$ is the element exhibited in Subsection \ref{SS:global expression} that defines $\pkernel{\F}$ then, using Lemma \ref{L:Cartier formula}, we can write
    \[
        \alpha (v^p ) = \frac{i^* \omega(v^p)}{\delta}  = \frac{i^*f}{\delta} \left( i^* \eta (v) \right)^p
    \]
    for every $v \in T_{\mathcal Y}$. Since the degeneracy divisor of the $p$-curvature of $\mathcal Y$ is the largest effective divisor $\pdegeneracy{\mathcal Y}$ such that $\pdegeneracy{\mathcal Y} \le (\alpha(v^p))_0$ for every $v \in T_{\mathcal Y}$ we deduce  that
    \[
        \pdegeneracy{\mathcal Y} = i^* (f)_0 - (\delta)_0 + p \cdot (i^* \eta)_0  = \restr{\pdegeneracy{\F}}{Y} - \different{\F}{Y} + p \cdot (i^*\eta)_0 \, .
    \]
    The result follows from Item (\ref{I:differentbis 2}) of Lemma \ref{L:differentbis2} .
\end{proof}

\subsection{The degeneracy divisor of the $p$-curvature on $\mathbb P^2_{\field}$ and $\mathbb P^1_{\field} \times \mathbb P^1_{\field}$}\label{SS:p-divisor}
In this subsection, we review a result on the structure of the degeneracy divisor of the $p$-curvature of sufficiently general foliations on $\mathbb P^2_{\field}$ and on Hirzebruch surfaces, and sketch its proof, obtained by the first author in \cite{Wodson}, see also \cite{WodsonTese}. We refer the reader to these works for details.

\begin{thm}\label{T:Wodson}
    Let $\field$ be an algebraically closed field of characteristic $p >0$. Let $X$ be equal to $\mathbb P^2_{\field}$ or a Hirzebruch surface. If $\F$ is a general foliation on $X$ with canonical bundle $\omega_{\F}$ then at least one of the following assertions holds true.
    \begin{enumerate}
        \item\label{I:not psef} The line-bundle $\omega_{\F}$ is not pseudo-effective and $\F$ is a $p$-closed foliation.
        \item\label{I:not nef} The surface is a Hirzebruch surface different from $\mathbb P^1_{\field} \times \mathbb P^1_{\field}$, the line-bundle $\omega_{\F}$ is pseudo-effective but not nef, the foliation $\F$ is $p$-dense but the unique curve of negative self-intersection appears as with multiplicity $p$ in the degeneracy divisor of the $p$-curvature of $\F$.
        \item\label{I:NF=0 mod p} The image of  $\omega_{\F}$ under the group morphism $ \Pic(X) \to \Pic(X)\otimes_{\mathbb Z} \mathbb F_p$ is zero.
        \item\label{I:reduced} The degeneracy divisor of the $p$-curvature of  $\F$ is reduced.
    \end{enumerate}
\end{thm}
\begin{proof}[Sketch of the proof]
    Let us first consider the case of $X = \mathbb P^2_{\field}$.  If $\F$ is a foliation of $\mathbb P^2_{\field}$ then the degree
    of $\F$, $\deg(\F)$, is defined the number of tangencies between of $\F$  a non-invariant line $\ell$, counted with multiplicities. By definition, $\deg(\F) \ge 0$  and a simple computation shows that $N_{\F} = \mathcal O_{\mathbb P^2}(\deg(\F) + 2 )$ and $\omega_{\F} = \mathcal O_{\mathbb P^2_{\field}}(\deg(\F) - 1))$.

    If $\deg(\F)=0$ then $\F$ is always $p$-closed. Indeed, if $\F$ is $p$-dense then the degeneracy divisor of the $p$-curvature would be
    a section of $N_{\F} \otimes \omega_{\F}^{\otimes p} = \mathcal O_{\mathbb P^2_{\field}}(2 - p)$ containing the singular set of $\F$.
    A clear absurdity for $p>2$. In the case $p=2$, we also reach a contradiction since $\sing(\F)$ is non-empty because $c_2(\Omega^1_{\mathbb P^2_{\field}}(2)) = 1$.

    If $\deg(\F)=1$ then $\F$ is defined by a regular vector field $v$ (since $\omega_{\F} = \mathcal O_{\mathbb P^2_{\field}}$).
    In suitable affine coordinates, where $0$ is a singularity of $v$ and the line at infinity is invariant,  a general $v$ takes the form
    \[
        v = \alpha x \frac{\partial}{\partial x}  +  \beta y \frac{\partial}{\partial y}  \, .
    \]
    If $\alpha/\beta \notin \mathbb F_p$ then $\F$ is $p$-dense with $\pdegeneracy{\F}$ equal to the sum of the lines $\{x=0\}$, $\{y=0\}$, and the line at infinity.

    If $\deg(\F)=2$ then we consider the foliation on $\mathbb P^2_{\field}$ defined, in affine coordinates, by
    \[
        \omega = (xdy -y dx) + \omega_2
    \]
    where $\omega_2$ is a general  homogeneous $1$-form with polynomial coefficients of degree two. The foliation defined by
    $\omega$ leaves the line at infinity invariant and three lines through $0 \in \mathbb A^2_{\field}$ defined by the vanishing of
    $\omega(R) = \omega_2(R)$, where $R= x \frac{\partial}{\partial x}  +   y \frac{\partial}{\partial y}$ is the radial (or Euler) vector field.
    It follows that the degeneracy divisor is supported on the union of four lines above and a curve $C$ of degree at most $p$, since the degeneracy divisor of the $p$-curvature of $\F$ has degree $p+4$. Blowing up the origin, one sees that the strict transform of $C$ must intersect the strict transform of a line through $0$ in exactly one point, and this implies, after some work,  that  for $\omega_2$ generic, $C$ is an irreducible and reduced curve of degree $p$ invariant by $\F$.

    For $\deg(\F)>2$, we start with a foliation of $\G$ on $\mathbb P^2_{\field}$ of degree two, with reduced $\pdegeneracy{\G}$ and
    leaving invariant three lines in general position, say $\{xyz=0\}$. We consider the endomorphism for every positive integer $\ell\ge 2$
    not divisible by $p$ one considers the endomorphism of $\mathbb P^2_{\field}$ defined by $\varphi_{\ell}(x:y:z)= (x^{\ell}: y^{\ell}:z^{\ell})$. The pull-back foliation $\F_{\ell} = \varphi_{\ell}^* \G$ is a foliation with $\omega_{\F_{\ell}} = \mathcal O_{\mathbb P^2_{\field}}(\ell)$, hence of degree $\ell+1$. Moreover, since $\varphi_{\ell}$ is étale on the complement of  the invariant divisor $\{ xyz=0\}$, Proposition \ref{P:pdegeneracy of pb} implies that $\pdegeneracy{\F_{\ell}}$ is a reduced divisor. This shows that when Assertions (\ref{I:not psef}), (\ref{I:not nef}) and (\ref{I:NF=0 mod p}) are not valid, then Assertion (\ref{I:reduced}) holds true.

    To prove the result on Hirzebruch surfaces, one construct examples of $p$-dense foliations with $p$-curvature having reduced degeneracy divisor  by applying a series of birational transformations and ramifications as above starting from examples on $\mathbb P^2_{\field}$. We refer to \cite{Wodson} for details.
\end{proof}

\begin{remark}\label{R:later use}
    We point out, for later use, that the proofs presented in \cite{Wodson} show that whenever there exists a foliation on $\mathbb P^2_{\field}$ or on a Hirzebruch surface with reduced degeneracy divisor for its $p$-curvature, there exists a foliation with the same property and same normal bundle that leaves invariant the boundary of $X$, seen as a toric surface.
\end{remark}

\section{Families of foliations}\label{S:deforming}
This section studies the behavior of the kernel of the $p$-curvature and the degeneracy divisor of the $p$-curvature under
deformations.

\subsection{Families of foliations}
Let $T$ be an irreducible algebraic variety. For us, a family of codimension $q$ foliations parametrized by $T$ consists of an algebraic variety $\mathscr X$, a smooth and projective morphism $\pi : \mathscr X \to T$, a line bundle $N_{\mathscr F}$ on $\mathscr X$, and a foliation $\mathscr F$ on $\mathscr X$ of codimension $q + \dim T$ defined by a relative $q$-form $\omega \in H^0(\mathscr X, \Omega^q_{\mathscr X/T} \otimes N_{\mathscr F})$ with codimension two singular set over any point $t \in T$ such that the kernel of
\begin{align*}
    T_{\mathscr X/T} & \longrightarrow \Omega^{q-1}_{\mathscr X/T} \otimes N_{\mathscr F} \\
    v & \mapsto i_v \omega
\end{align*}
is an involutive subsheaf of $T_{\mathscr X}$ of rank $\dim \mathscr X - \dim T - q$.

When $\mathscr X = X \times T$, we will say that $\mathscr F$ is a family of foliations on $X$.

\begin{remark}
    There are other possibilities for defining a family of foliations parameterized by $T$. One can, for instance, impose that a family of foliations is defined by an involutive subsheaf of $T_{\mathscr X/T}$ which is flat over $T$. This would lead to a more stringent concept for a family of foliations, which many authors have already studied. Our definition above aims to study the irreducible components of the space of foliations on projective spaces as considered by Cerveau, Lins Neto, and others, see \cite{MR1394970, MR4354288} and references therein.
\end{remark}

\subsection{Behavior of the kernel of the $p$-curvature under deformations}
If $X$ is a polarized smooth and connected projective variety of dimension $n$, with polarization given by an ample divisor $H$, and $\mathcal E$ is a reflexive sheaf on $X$, we set the degree of $\mathcal E$ with respect to the polarization $H$ equal to $\det(\mathcal E) \cdot H^{n-1}$.

Recall from the statement of Lemma \ref{L:kernel pcurvature} that $\pkernel{\F}$ is the distribution with tangent sheaf determined by the kernel of the $p$-curvature of a $p$-dense foliation $\F$, \ie,
\[
    \Frob^* T_{\pkernel{\F}} = \ker \psi_{\F} \, .
\]
Naturally, $\omega_{\pkernel{\F}}= \det (T_{\pkernel{\F}})^*$  is the canonical bundle of $\pkernel{\F}$.

\begin{lemma}\label{L:upperbound}
    Let $\mathscr F$ be a family of codimension one foliations on a smooth projective variety $X$ parameterized by an
    algebraic variety $T$. If $0 \in T$ is a closed point such that $\mathscr F_0$ is $p$-dense then there exists an open neighborhood
    $U \subset T$ of $0$ such that for every closed point $t \in U$, the foliation $\mathscr F_t$ is $p$-dense and the inequality
    \[
         \deg_H \omega_{\pkernel{\mathscr F_t}} \le \deg_H \omega_{\pkernel{\mathscr F_0}}
    \]
    holds true for any polarization $H$.
\end{lemma}
\begin{proof}
    Consider the relative $p$-curvature morphism
    \begin{align*}
        \pcurvature{\mathscr F/T} : \Frob^* T_{\mathscr F/ T} & \longrightarrow N_{\mathscr F} \\
        v&\mapsto \omega(v^p) \, ,
    \end{align*}
    where $\Frob : \mathscr X \to \mathscr X$ is the absolute Frobenius.

    Let $\mathcal C$ be the support of $\coker \pcurvature{\mathscr F/T}$ and fix an embedding of $X$ in some $\mathbb P^N$. According to \cite[Chapter III, Theorem 9.9]{hartshorne1977},  for every $m \gg 1$ the Hilbert function $h_{\mathcal C}(m,t)=  \chi(X, \mathcal O_{\mathcal C_t}(m))$ is lower semi-continuous, where
    $\mathcal O_{\mathcal C_t}(m) = \mathcal O_{\mathcal C_t} \otimes \mathcal O_{\mathbb P^N}(m)$ as usual.

    Let $Z$ be the subset of $T$ such that for every $z \in Z$, corresponding to the $p$-closed foliations in the family. For every $t$ in the complement $U= T-Z$, the subscheme $\mathcal C_t$ of $X$ is different from $X$ and has divisorial part equal to $\pdegeneracy{\mathscr F_t}$.
    There exists a positive constant $c$ such that  $c \cdot \deg(\pdegeneracy{\mathscr F_t})$ is the leading coefficient of the Hilbert function $h_{\mathcal C}(m,t)$.

    Since $\det( \ker \pcurvature{\mathscr F_t} )^*$, the Frobenius pull-back of the canonical bundle of the kernel of the $p$-curvature of $\mathscr F_t$, is isomorphic to
    $\det( \Frob^* T_{\mathscr X_t})^* \otimes \det( N_{\mathscr F_t} \otimes \mathcal I_{{\mathcal C}_t} )$, we deduce from
    the lower semi-continuity of $h_{\mathscr C}$, the upper-semicontinuity of the  function $t \mapsto \deg(\det( \ker \pcurvature{\mathscr F_t} ))$  on $U$ as wanted.
\end{proof}

If one further assumes that the degeneracy divisor of a $p$-dense codimension one foliation is free from $p$-th powers then
one gets the reverse inequality for the degree of the kernel of the $p$-curvature.

\begin{lemma}\label{L:constant pkernel}
    Let $\mathscr F$ be a family of codimension one foliations on a smooth projective variety $X$ parameterized by an
    algebraic variety $T$. If $0 \in T$ is a closed point such that $\mathscr F_0$ is $p$-dense and the degeneracy divisor of
    its $p$-curvature is free from $p$-th powers then there exists a non-empty open subset
    $V \subset T$ such that for every closed point $t \in U$, the foliation $\mathscr F_t$ is $p$-dense and the equality
    \[
         \deg_H \omega_{\pkernel{\mathscr F_t}} = \deg_H \omega_{\pkernel{\mathscr F_0}}
    \]
    holds true for any polarization $H$.
\end{lemma}
\begin{proof}
    Thanks to Lemma \ref{L:upperbound}, it suffices to show the existence of an open neighborhood $U$ of $ 0 \in T$ such that
    \[
         \deg_H \omega_{\pkernel{\mathscr F_t}} \ge \deg_H \omega_{\pkernel{\mathscr F_0}}
    \]
    for any $t \in U$. For that, let $\mathcal L$ be a very-ample line bundle on $\mathscr X = X \times T$ such that $T_{\mathscr F} \otimes \mathcal L$ is generated by global sections. Choose $v \in H^0(\mathscr X, T_{\mathscr F} \otimes \mathcal L)$ such that $v_0$, the restriction of $v$ to $X \times \{0 \}$, is generically transverse to $\pkernel{\mathscr F_0}$.
    If $\omega \in H^0(\mathscr X, \Omega^1_{\mathscr X/T} \otimes N_{\mathscr F})$ is a twisted $1$-form defining the family of foliations
    $\mathscr F$ then the section $\sigma = \omega( v^p) \in H^0(\mathscr X, N_{\mathscr F} \otimes \mathcal L^{\otimes p})$ defines a divisor $\mathscr D$ on $\mathscr X$. The choice of $v$ guarantees that  the support of $\mathscr D$ does not contain $X_0$.
    Consequently, there exists an open neighborhood of $0 \in T$ such that the support of $\mathscr D$ does not contain $\mathscr X_t$.
    For any $t \in U$, $\mathscr{D}_t := \restr{\mathscr{D}}{\mathscr X_t}$  is, by construction, equal to $\pdivisor{\mathscr F_t} + p R_t$ for some effective divisor $R_t$.  Let  $\mathscr R \in \Div(\mathscr X)$ be the maximal effective divisor such that   $p \mathscr R \le \mathscr D$. Semi-continuity implies that $\mathscr R_t := \restr{\mathscr R}{ \mathscr X_t}$ is equal to $R_t$ for every $t$ in an non-empty open subset $V$ of $T$ (not necessarily containing $0$).

    Since we are assuming that  $\pdivisor{\mathscr F_0} = \mathscr D_0 - p R_0$ is free from $p$-th powers, we have that $\mathscr R_0 \le R_0$. Using that the degrees $\deg_H \mathscr D_t$ and $\deg_H \mathscr R_t$ do not depend on $t$, we obtain that
    \begin{align*}
        \deg_H \pdivisor{\mathscr F_0} & \le \deg_H \left( \pdivisor{\mathscr F_0}  + p (R_0 -  \mathscr R_0) \right)       \\
        & = \deg_H \left( \mathscr D_0 \right)  - \deg_H\left( p \mathscr R_0 \right)   \\
        & = \deg_H \left( \mathscr D_t \right)  - \deg_H\left( p \mathscr R_t \right)  = \deg_H \pdivisor{\mathscr F_t}
    \end{align*}
    for every $t \in V$. Proposition \ref{P:degree degeneracy} implies the result.
\end{proof}

\section{Integrability of the Cartier transform}

So far, we have studied foliations in positive characteristic without reference to foliations in characteristic zero. In this section, we show that foliations in positive characteristic obtained through reduction of foliations in characteristic zero are far from being arbitrary as the Cartier transform of them are integrable/involutive, a property that does not hold for arbitrary foliations.

\subsection{Lifts of foliations modulo $p^2$}
Let $\field$ be an algebraically closed field of characteristic $p>0$,  let $\Witt{2}{\field}$ be the ring of Witt vectors of length $2$ with
values in $\field$, and set $S= \Spec(\Witt{2}{\field})$. A lift of a smooth algebraic variety $X$ module $p^2$ is a flat smooth $S$-scheme $X'$ such that $X = X' \otimes \field$. Likewise a lift modulo $p^2$ of a foliation on $X$ consists of a lift $X'$ of $X$ modulo $p^2$ and subsheaf
$T_{\F'} \subset T_{X'/S}$, flat over $S$, such that
\begin{enumerate}
    \item the sheaf $T_{\F'}$ is a lift of $T_{\F}$, \ie, $T_{\F} = T_{\F'} \otimes \field$; and
    \item the sheaf $T_{\F'}$ is involutive, \ie the morphism
    \[
        \bigwedge^2 T_{\F'} \to \frac{T_{X'/S}}{T_{\F'}}
    \]
    induced by Lie brackets is identically zero.
\end{enumerate}

As already noted by Miyaoka in \cite[Example 4.2]{MR927960} not every foliation admits a lift modulo $p^2$. As it will be essential in what
follows, we present below a variant of Miyaoka's example.

\begin{ex}\label{Ex:Miyaoka}
    Let $p$ be an odd prime, let $\field$ be the algebraic closure of $\mathbb F_p$, and let $\omega$ be the $1$-form $x^{p-1} dx + z^p y^{p-1}dy$ on $\mathbb A^3_{\mathbb Z}$. The reduction modulo $p$ of $\omega$ defines a foliation $\F$ on $\mathbb A^3_{\field}$
    that does not admit a lift modulo $p^2$.
\end{ex}
\begin{proof}
    First observe that
    \[
        \omega \wedge d \omega = p(xyz)^{p-1} dx \wedge dy \wedge dz
    \]
    and, therefore, the reduction modulo $p$ of $\omega$ defines an involutive subsheaf of $T_{\mathbb A^3_{\field}}$. Let $\F$ be the corresponding foliation on $\mathbb A^3_{\field}$. If $\F$ admits a lift modulo $p^2$ and we denote by $\overline \omega$
    the reduction of $\omega$ modulo $p^2$ then such lift is defined by the $1$-form $\overline \omega + p \eta$ for a suitable  $1$-form
    $\eta \in \Omega^1_{\mathbb A^3_{\field}}$ which satisfies
    \begin{equation}\label{E:integrabibility p2}
        (\overline{\omega} + p \eta) \wedge d (\overline{\omega} + p \eta) = 0 \quad \text{in} \quad \Omega^1_{\mathbb A^3_{\Witt{2}{\field}}}.
    \end{equation}
    Expanding the lefthand-side  of Equation (\ref{E:integrabibility p2}), we get
    \[
        p\left(z^py^{p-1}dy\wedge d\eta  +(xyz)^{p-1}dx\wedge dy \wedge dz +x^{p-1}dx\wedge d\eta\right) = 0
    \]
    Since $d \eta$ has no terms that are scalar multiples of $(xy)^{p-1}$, this expression is always non-zero. It follows that
    $\F$ does not admit a lift modulo $p^2$.
\end{proof}

\begin{remark}
    Observe that the arguments presented in Example \ref{Ex:Miyaoka}  not only show that $\F$ does not admit a lift modulo $p^2$ but also
    show that the restriction of $\F$ to any open subset $U \subset \mathbb A^3_{\field}$ containing the origin does not admit a lift modulo $p^2$.
\end{remark}

\subsection{Integrability of the Cartier transform of liftable foliations}
Our next result is an amplification of Miyaoka's Example \ref{Ex:Miyaoka}.

\begin{thm}\label{T:lift}
    Let $\F$ be a codimension one foliation on a $X$ smooth projective variety defined over a field
    $\field$ of characteristic $p>2$. If $\F$ if $p$-dense and lifts to a foliation over the
    ring of Witt vectors $\Witt{2}{\field}$ then the Cartier transform of $\mathcal F$ is a foliation.
\end{thm}
\begin{proof}
    We will prove the contra-positive, \ie if the Cartier transform of a $p$-dense codimension one foliation $\F$ is not a foliation then $\F$ does not lift modulo $p^2$.

    Proposition \ref{P:closed} guarantees the existence of a closed rational $1$-form  $\omega$  defining $\F$.
    The Cartier transform of $\F$ is defined by the rational $1$-form
    $\theta = \Cartier(\omega)$. By assumption, the $3$-form $\theta \wedge d \theta$ does not vanish identically. In particular,
    $\dim X \ge 3$.

    Assume first that $\dim X=3$. Let $a \in X$ be a closed point such that $\theta$ is regular (no poles) at $x$ and
    $\theta \wedge d\theta (x) \neq 0$. Denote by $\mathfrak m \subset \mathcal O_{X,a}$ the corresponding maximal ideal.
    Let $x,y \in \mathfrak m$ be such that $d\theta = dx\wedge dy \mod \mathfrak m \Omega^2_{X,a}$.  We can {\it integrate} the $2$-form
    $d\theta$ modulo $\mathfrak m^2$ and  write
    \[
        \theta = x dy + z dz  +  w^{p-1} dw \mod \mathfrak m^2 \Omega^1_{X,a}
    \]
    for some $z,w  \in \mathfrak m$. If $p>2$ then we can assume that $w$ is equal to zero  as $w^{p-1} \in \mathfrak m^2$. The non-vanishing of
    $\theta\wedge d \theta$ at $a$ implies that $x,y,z$ generate $\mathfrak m$.

    Since $\theta = \Cartier(\omega)$, we can write
    \[
        \omega = x^p y^{p-1} dy + z^{p-1} dz + df \mod \mathfrak m^2 \Omega^1_{X,a}
    \]
    for some $f \in \mathfrak m$.

    Let  $X'$ be a lift of $X$, let $a'$ be a lift of $a$, and let $\omega'$ be a $1$-form in $\Omega^1_{X',a'}$  that defines a lift $\F'$ of $\F$ to $X'$ and with reduction modulo $p$ equal to $\omega$. We can write
    \[
        \omega' = (x^py^{p-1}dy + z^{p-1}dz + df') + p \eta   \mod \left< p^2, \mathfrak m^{2p} \Omega^1_{X',a'}  \right> \, ,
    \]
    for any lift $f'$ of $f$ and a certain $\eta \in \Omega^1_{X,a}$. Expanding $\omega' \wedge d\omega'$ we get
    \begin{align*}
        & p \left( ( x^p y^{p-1}  dy + z^{p-1} dz + df') \wedge d\eta + (xyz)^{p-1} dx \wedge dy \wedge dz \right) + \\
       &+ p \left( (xy)^{p-1} df'\wedge dx \wedge dy  \right)
    \end{align*}
    modulo  $\left< p^2, \mathfrak m^{3p-2} \Omega^1_{X',a'}  \right>$. As in Example \ref{Ex:Miyaoka}, for no matter which choice of $f'$ and $\eta$, the term  $(xyz)^{p-1} dx \wedge dy \wedge dz$ will not be cancelled. This shows that no lift of $\omega$ defines a foliation when $\dim X=3$.

    When $\dim X>3$, we reduce to the previous case by restricting to the intersection of $\dim X - 3$ sufficiently general hyperplane sections.
\end{proof}

As pointed out by the referee, thanks to Chow's Lemma  \cite[\href{https://stacks.math.columbia.edu/tag/0200}{Tag 0200}]{stacks-project},  the result above also holds for proper smooth varieties.

\section{Lifting the kernel of the $p$-curvature}\label{S:lift}
This section defines integral models for holomorphic foliations and discusses the lifting, back to characteristic zero, of properties of the reduction to positive characteristic of a holomorphic foliation.

\subsection{Reduction modulo $p$ of holomorphic foliations on projective manifolds}
Let $\F$ be a singular holomorphic foliation defined on a complex projective manifold $X$. The variety $X$ and the subsheaf $T_\F \subset T_X$ can be viewed as objects defined over a finitely generated $\mathbb Z$-algebra $R \subset \mathbb C$. More precisely, there exists a projective scheme $\mathscr X$ flat over $\Spec(R)$ and subsheaf $T_{\mathscr F} \subset T_{\mathscr X/R}$ flat over $\Spec(R)$ such that $X$ is isomorphic to the base change $\mathscr X \otimes_{R} \mathbb C$ and $\mathcal F$ is isomorphic to $\mathscr F \otimes_R \mathbb C$.

The pair $(\mathscr X, \mathscr F)$ is called an integral model for $(X,\F)$ and is not uniquely determined by $(X,\mathcal F)$. Roughly speaking, it depends on the choice of generators for the ideal of $X$, on a finite presentation of $T_{\F}$,  as well as on the choice of an open subset of the affine scheme determined by the coefficients used to define the ideal and the presentation.

If $\mathfrak p \subset R$ is a maximal ideal, then the residue field $\field(\mathfrak p) = R/\mathfrak p$ is a finite field of characteristic $p>0$ with algebraic closure $\field = \overline{\field(\mathfrak p)} \simeq \overline{\mathbb F_p}$. The reduction of $\F$ modulo $\mathfrak p$ is the foliation $\F_{\mathfrak p}$ of $X_{\mathfrak p} = \mathscr X \otimes_R \field$ determined by $T_{\F_{\mathfrak p}} = T_{\mathscr F} \otimes_R \field$.

\begin{ex}\label{L:logvaries}
    Let  $\F$ be the codimension one foliation on $\mathbb A^3_{\mathbb C}$ defined by the $1$-form
    \[
        \omega = \sqrt{-1} \cdot \frac{dx}{x} +  \frac{dy}{y} + \frac{dz}{z} \, .
    \]
    We can take  the $\mathbb Z$-algebra $R =\mathbb Z[\sqrt{-1}] \simeq \mathbb Z[\alpha]/(\alpha^2 +1)$ as the ring of definition of $\mathcal F$. The ring $R$ has the following (maximal) prime ideals
    \begin{enumerate}
        \item\label{I:over2} only one prime ideal over $2$, namely $(1+\sqrt{-1})R=(1-\sqrt{-1})R$;
        \item\label{I:1mod4} for each rational prime $p=a^2 + b^2$ congruous to $1$ modulo $4$, two distinct prime ideals, namely $(a \pm \sqrt{-1}b) R$; and
        \item\label{I:3mod4} for each rational prime $p$ congruous to $3$ modulo $4$, the prime ideal $pR$.
    \end{enumerate}

    If we reduce $\F$ modulo $(1+\sqrt{-1})R$, we get a $p$-closed foliation over $\overline{\mathbb F}_2$ defined by $d \log (xyz)$.
    Likewise, if we reduce $\F$ modulo $(a\pm b \sqrt{-1})R$, $p=a^2+b^2$ being a prime congruous to $1$ modulo $4$, we get a $p$-closed foliation
    over $\overline{\mathbb F}_p$ defined by $d \log(xyx)$. In contrast, if we reduce $\F$ modulo $\mathfrak p = pR$, where  $p$ is a rational prime congruous to $3$ modulo $4$, we get a $p$-dense foliation with $\pkernel{\mathscr F_{\mathfrak p}}$ defined by the kernel of the $2$-form
    \[
        \omega_{\mathfrak p} \wedge \Cartier( \omega_{\mathfrak p} ) = (\alpha_{\mathfrak p} \cdot \frac{dx}{x} +  \frac{dy}{y} + \frac{dz}{z}) \wedge ( \alpha_{\mathfrak p}^{1/p}  \cdot \frac{dx}{x} +  \frac{dy}{y} + \frac{dz}{z}) \, ,
    \]
    where $\alpha_{\mathfrak p}$ stands for the class of $\alpha$ in the algebraic closure of $R/\mathfrak p \simeq \mathbb F_{p^2}$.
\end{ex}

\begin{ex}
    Let  now $\F$ be the codimension one foliation on $\mathbb A^3_{\mathbb C}$ defined by the $1$-form
    \[
        \omega = \alpha \cdot \frac{dx}{x} +  \frac{dy}{y} + \frac{dz}{z} \, ,
    \]
    with $\alpha$ equal to an arbitrary transcendental number.
    We can take the $\mathbb Z$-algebra $R=\mathbb Z[\alpha]$ as the ring of definition of $\mathcal F$.
    The maximal ideals of the ring $R$ are of the form $(p,f(\alpha)) R $ where $p$ is a rational prime
    and $f \in \mathbb Z[\alpha]$ is a monic polynomial with irreducible reduction modulo $p$.
    The reduction of $\mathcal F$ modulo $\mathfrak p = (p,f(\alpha))  R$ is $p$-closed if the degree of $f$ is one, otherwise
    $\mathscr F_{\mathfrak p}$ is $p$-dense.
\end{ex}

\subsection{A conjecture by Ekedahl, Shepherd-Barron, and Taylor}

In general, it is not easy to draw conclusions about the geometry of the holomorphic foliation $\F$ from the behavior of the
reduction modulo $\mathfrak p$ for different primes $\mathfrak p$.

\begin{conj}[Ekedahl, Shepherd-Barron, and Taylor]\label{Conj:ESBT}
    Let  $\F$ be a holomorphic foliation on a complex projective manifold $X$ with an integral model $(\mathscr X, \mathscr F)$ defined over a finitely generated $\mathbb Z$-algebra
    $R$. The foliation $\F$ is algebraically integrable (\ie all its leaves are algebraic) if, and only if, $\mathcal F_{\mathfrak p}$ is $p$-closed for every maximal prime (\ie closed point) $\mathfrak p$ in a non-empty Zariski open subset of $\Spec(R)$.
\end{conj}

Despite the important particular cases settled by Bost in \cite{MR1863738}, Conjecture \ref{Conj:ESBT} is wide open.  It must be mentioned  that Conjecture \ref{Conj:ESBT} generalizes a conjecture by Grothendieck and Katz, which is also wide open despite recent advances \cite{MR3827814,MR4270046}.

Under the assumptions of  Conjecture \ref{Conj:ESBT}, Theorem \ref{T:pintegra} implies  the existence of
algebraic subvarieties invariant by $\mathscr F_{\mathfrak p}$ that cover $X_{\mathfrak p}$ for a Zariski open and dense set of maximal primes $\mathfrak p \subset \Spec(R)$. This is not enough to guarantee the existence of algebraic subvarieties invariant by $\mathcal F$ covering $X$ as the degrees of the subvarieties of $X_{\mathfrak p}$ invariant by $\mathscr F_{\mathfrak p}$ may be unbounded as a function of $\mathfrak p$.

\subsection{Lifting subvarieties/subdistributions of bounded degree}

If one can guarantee the existence of subvarieties/subdistributions for the reduction of a foliation modulo a Zariski dense set of maximal primes with degree uniformly bounded then the lemma below implies the existence of the same type of objects in characteristic zero respecting the same bounds.

\begin{lemma}\label{L:lifting hyp sub}
    Let $\F$ be a holomorphic foliation on a polarized projective manifold
    $(X, H)$ and let $(\mathscr X, \mathscr H, \mathscr F)$ be an integral model for $(X, H,\F)$ with ring of definition $R$.
    If there are integers $M,m$, and a Zariski dense set of maximal primes
    $\mathscr P \subset \Spec(R)$ such
    $\mathscr F_{\mathfrak p}$ has, for every $\mathfrak p \in \mathscr P$,
    \begin{enumerate}[label=(\arabic*)]
        \item\label{I:lift subvariety H} an invariant subvariety of dimension $m$ and degree at most $M$; or
        \item\label{I:lift subfoliation H} a subdistribution $\mathscr G_{\mathfrak p}$ of dimension $m$ with $\deg_{\mathscr H_{\mathfrak p}} \omega_{\mathscr G_{\mathfrak p}} \le M$
    \end{enumerate}
    then $\F$ has, respectively,
    \begin{enumerate}[label=(\roman*)]
        \item\label{I:lift subvariety C}  an invariant subvariety of dimension $m$ and degree at most $M$; or
        \item\label{I:lift subfoliation C} a subdistribution $\mathcal G$ of dimension $m$ with $\deg_{H} \omega_{\mathcal  G} \le M$ .
    \end{enumerate}
\end{lemma}
\begin{proof}
    The implication  $\ref{I:lift subvariety H} \implies \ref{I:lift subvariety C}$ is the content of \cite[Proposition 7.1]{MR3842065}.
    The proof of the implication $\ref{I:lift subfoliation H} \implies \ref{I:lift subfoliation C}$  follows the same principles and is very similar to the proof of  \cite[Proposition 6.2]{MR3167130}. We proceed to give details of this later implication.

    Observe that subdistributions of ${\mathcal F}$, resp.  $\mathscr F_{\mathfrak p}$, are in one-to-one correspondence with torsion-free quotients of $T_{\F}$, resp. $T_{\mathscr F_{\mathfrak p}}$. Once the degree of the quotient is bounded from above,  \cite[Corollary 2.3]{MR1611822} guarantees that the possible Hilbert polynomials for the quotients belong to a finite set. Moreover, bounds for the degree of the quotient are equivalent, through adjunction, to bounds for the degree of the canonical bundle of the subdistribution.

    The formation of  Quot schemes commute with base change, \cite{MR1611822}. Therefore, for a fixed Hilbert polynomial $\chi$, the scheme $\Quot_{\chi}(X,T_{\F})$ is non-empty if, and only if, its reduction modulo $\mathfrak p$,  $\Quot_{\chi}(X_{\mathfrak p}, T_{\mathscr F_{\mathfrak p}})$, is non-empty for a Zariski dense set of primes $\mathfrak p$. This is sufficient to prove that $\ref{I:lift subfoliation H} \implies \ref{I:lift subfoliation C}$.
\end{proof}

To obtain a priori upper bounds for the degree of invariant subvarieties of the reduction to positive characteristic does not seem more manageable than obtaining the same upper bounds in characteristic zero. In contrast, if $\F$ is a holomorphic foliation such that the reduction modulo $\mathfrak p$ is not $p$-closed for a Zariski dense set of primes $\mathfrak p \subset \Spec(R)$ then the degree of the canonical bundle of $\pkernel{\mathscr F_{\mathfrak p}}$ is uniformly bounded as a function of $\mathfrak p$. Therefore one gets from the lemma above the existence of certain subdistributions in characteristic zero. In Section \ref{S:components char 0} we will explore this fact to obtain new results about the irreducible components of the space of holomorphic codimension one foliations on projective spaces.

\subsection{A cautionary remark}
For foliations of codimension one, Conjecture \ref{Conj:ESBT} can be rephrased as follows:
\begin{quote}
    the general leaf of a codimension one holomorphic foliation $\F$ is Zariski dense if, and only if, $\F_{\mathfrak p}$ is $p$-dense for a non-empty Zariski open set of maximal primes $\mathfrak p \in \Spec(R)$.
\end{quote}

Dropping the codimension one assumption, one obtains a strengthening of Conjecture \ref{Conj:ESBT} which turns out to be false as we now proceed to
show.

\begin{prop}\label{P:cautionary}
    Let $X$ be a complex projective manifold of dimension $2n+1$. Assume that $H^0(X,\Omega^2_X)$ is  generated by a $2$-form $\omega$ satisfying $\omega^n\neq0$, and  let $\F$ be the foliation on $X$ of dimension one defined by the kernel of $\omega$. If $(\mathscr X, \mathscr F)$ is an integral   model for $(X,\F)$ defined over a finitely generated $\mathbb Z$-algebra $R$ then $\F_{\mathfrak p}$ is not $p$-dense for any maximal prime
    $\mathfrak p$ in an open and dense subset of $\Spec(R)$.
\end{prop}
\begin{proof}
    By semi-continuity, there exists an open and dense subset $U \subset \Spec(R)$ such that $h^0(\mathscr X_{\mathfrak p}, \Omega^2_{\mathscr X_{\mathfrak p}})=1$ for
    every maximal prime $\mathfrak p \in U$. For $\mathfrak p \in U$, let $\omega_{\mathfrak p}$ be the generator of $H^0(\mathscr X_{\mathfrak p}, \Omega^2_{\mathscr X_{\mathfrak p}})$ or,
    equivalently, the reduction modulo $\mathfrak p$ of a generator of $H^0(\mathscr X, \Omega^2_{\mathscr X})$. Replacing $U$ by an open subset, we can assume that $\omega_{\mathfrak p}^n\neq 0$  for every $\mathfrak p \in U$.

    Fix an arbitray maximal prime $\mathfrak p \in U$. Since $\omega_{\mathfrak p}^n\neq 0$, the kernel  of the morphism from $T_{X_{\mathfrak p}}$ to $\Omega^1_{X_{\mathfrak p}}$ defined by contraction with $\omega_{\mathfrak p}$ is the tangent sheaf of a foliation $\F_{\mathfrak p}$ of dimension one. Let $v$ be any rational section of $T_{\F_{\mathfrak p}}$. Since $h^0(X,\Omega^2_{X_{\mathfrak p}})=1$,
    $\Cartier(\omega_{\mathfrak p})$ is a multiple of $\omega_{\mathfrak p}$. Consequently $i_v \Cartier(\omega_{\mathfrak p}) = 0$. According to Lemma \ref{L:Cartier formula},
    \[
        0 = i_v \Cartier(\omega_{\mathfrak p}) = \Cartier(i_{v^p} \omega_{\mathfrak p} - (L_v)^{p-1}(i_v\omega_{\mathfrak p})) = \Cartier(i_{v^p} \omega_{\mathfrak p}).
    \]
    The vanishing of $ \Cartier(i_{v^p} \omega)$ implies the existence of a rational function $f$ on $X_{\mathfrak p}$ such that
    $i_{v^p} \omega_{\mathfrak p} = df$, \ie $i_{v^p} \omega_{\mathfrak p}$ is exact. If $df=0$ then $\F_{\mathfrak p}$ is $p$-closed, if not then $f$ is a non-trivial rational first
    integral for $\F_{\mathfrak p}$. In both cases, $\F_{\mathfrak p}$ is not $p$-dense as claimed.
\end{proof}

If $X$ is a complex projective symplectic manifold of dimension $2n$ with symplectic form $\omega$ and $Y\subset X$ is a smooth hypersurface with
big normal bundle then the characteristic foliation of $Y$, \ie the one-dimensional foliation on $Y$ induced by the pull-back of $\omega^{n-1}$ to $Y$, has Zariski dense general leaf \cite[Theorem 1.7]{https://doi.org/10.48550/arxiv.2102.02799}. In contrast, Proposition \ref{P:cautionary} above implies that any integral model for $\F$ is not $p$-dense for almost every prime $\mathfrak p$.

\part{The space of holomorphic foliations on projective spaces}

\section{Foliations on complex projective spaces}\label{S:components char 0}

In the second part of the paper, we will focus on codimension one foliations on complex projective spaces.
We will  use the results presented in the first part in order to obtain new information about the irreducible
components of the scheme $\Fol{d}{n}{\C}$ parameterizing degree $d$ codimension one foliations on $\mathbb P^n_{\C}$.

In this section, we settle the notation, recall the definition of the space of foliations on projective spaces, and
establish the conventions concerning the reduction to positive characteristic used throughout the second part of the paper.

\subsection{Foliations on projective spaces}
Let $\field$ be an arbitrary algebraically closed field and let $\F$ be a codimension $q\ge 1$ foliation
on the projective space $\mathbb P^n_{\field}$, $n \ge q+1$. Let $\omega \in H^0(\mathbb P^n_{\field}, \Omega^q_{\mathbb P^n_{\field}}\otimes \det \NF)$ be a twisted $q$-form with coefficients in the line-bundle $\det N_{\F}$ such that $T_{\F} \subset  T_{\mathbb P^n_{\field}}$ equals to the kernel of morphism
\[
    T_{\mathbb P^n_{\field}} \longrightarrow  \Omega^{q-1}_{\mathbb P^n_{\field}}\otimes \det \NF
\]
defined by contraction with $\omega$.  The degree of $\F$ is, by definition, the
degree of the zero locus of $i^*\omega \in H^0(\mathbb P^q_{\field}, \Omega^q_{\mathbb P^q_{\field}} \otimes i^* \NF)$, where $i: \mathbb P^q_{\field} \to \mathbb P^n_{\field}$ is a general linear embedding of $\mathbb P^q_{\field}$ into $\mathbb P^n_{\field}$. In order words,
the degree of $\F$ is the degree of tangency divisor of $\F$ with general $\mathbb P^q_{\field}$ linearly embedded in $\mathbb P^n_{\field}$.
Since $\Omega^q_{\mathbb P^q_{\field}} = \mathcal O_{\mathbb P^q_{\field}}(-(q+1))$, it follows that
\[
    \det N_{\F} = \mathcal O_{\mathbb P^n_{\field}}(\deg(\F) + q + 1) \,
\]
for any codimension $q$ foliation $\F$ of degree $\deg(\F)$. Adjunction formula implies that any such foliation $\F$
has canonical sheaf satisfying
\[
    \omega_{\F} = \mathcal O_{\mathbb P^n_{\field}}(\deg(\F) - \dim(\F)).
\]

It follows from  Euler's sequence, that we can identify $H^0(\mathbb P^n_{\field}, \Omega^q_{\mathbb P^n_{\field}}(d+q+1))$ with the vector space of homogenous polynomial $q$-forms
\[
     \sum_{0\le i_1 \le \ldots \le i_q \le n}^n a_{i_1, \ldots, i_q}(x_0, \ldots, x_n) dx_{i_1}\wedge \cdots \wedge dx_{i_q}
\]
with coefficients $a_{i_1, \ldots, i_q} \in \field[x_0, \ldots, x_n]$ of degree $d+1$,  which are annihilated by the radial/Euler vector field $R = \sum_{i=0}^n x_i \frac{\partial}{\partial x_i}$. We will say that a homogeneous polynomial $q$-form on $\mathbb A^{n+1}_{\field}$ representing an element of $H^0(\mathbb P^n_{\field}, \Omega^q_{\mathbb P^n_{\field}}(d+q+1))$ is a projective $q$-form of degree $d+q+1$.

\subsection{The space of codimension one foliations on projective spaces}
The space/scheme of codimension one foliations of degree $d$ on the projective space $\mathbb P^n_{\field}$ is, by definition, the
locally closed subscheme $\Fol{d}{n}{\field}$ of $\mathbb P H^0(\mathbb P^n_{\field}, \Omega^1_{\mathbb P^n_{\field}}(d+2))$
defined by the conditions
\[
    [\omega] \in \Fol{d}{n}{\field} \text{ if, and only if, } \codim \sing(\omega) \ge 2 \text{ and }  \omega \wedge d \omega = 0\,.
\]
The space/scheme of codimension $q$ foliations of degree $d$ foliations of $\mathbb P^n_{\field}$ can be defined similarly, we
refer to \cite{MR2405162} for details. Here we will only mention that besides the integrability/Frobenius condition, one also has to add
the decomposability/Pl\"ucker conditions  to consider only $q$-forms with kernels having corank $q$.

\subsection{Conventions about the reduction to positive characteristic}
Let $\F$ be a codimension one foliation on $\mathbb P^n_{\C}$, \ie $\F$ is a codimension one holomorphic foliation
of a complex projective space of dimension $n$. Throughout this part of the paper, we will need to consider an integral model
$(\mathscr X, \mathscr F)$ for $\F$ as defined in Section \ref{S:lift}. Concretely, if $\omega$ is a projective $1$-form
defining $\F$ and  $R_0 \subset \C$ is any finitely generated $\mathbb Z$-algebra containing the coefficients of $\omega$ then
we take: $R$ equal to a finitely generated extension of $R_0$ obtained by inverting finitely many elements of $R_0$, $\mathscr X = \mathbb P^n_{R} = \Proj R[x_0, \ldots, x_n]$, and $\mathscr F$ as the foliation defined by  $\omega$. Notice that one needs to pass to the open subset
$\Spec(R) \subset \Spec(R_0)$   to obtain the flatness properties required by the definition.

We will say that a property $P$ for a foliation $\F$ holds for almost every prime if, and only if, there exists an integral model of $\F$ defined over a finitely generated $\mathbb Z$-algebra $R$ and there exists a non-empty open subset $U \subset \Spec(R)$ such that for every maximal prime $\mathfrak p \in U$, the property $P$ holds for the reduction of  $\F$ modulo $\mathfrak p$.

Likewise, we will say a property $P$ holds for a Zariski dense set of primes if, and only if, there exists an integral model of $\F$ defined over a finitely generated $\mathbb Z$-algebra $R$ and there exists a Zariski dense subset $Z \subset \Spec(R)$ formed only by maximal primes such that for every $\mathfrak p \in Z$, the property $P$ holds for the reduction modulo $\mathfrak p$ of
$\mathscr F$.

These two notational shortcuts will allow us to shorten the statements quite a bit.

\section{Subdistributions of minimal degree}
If $\F$ is a codimension one foliation on $\mathbb P^n_{\field}$ and $i$ is an integer between $1$ and $n-1$ then we define
\[
    \distmin{i}{\F} = \min_{\mathcal D \subset \F}  {\deg(\mathcal D)} \, ,
\]
where $\mathcal D$ ranges over all codimension $i$ distributions everywhere tangent to in $\F$.
By definition, $\distmin{1}{\F}  = \deg(\F)$.

If $\Sigma \subset \Fol{d}{n}{\field}$ is an irreducible component then we will set $\distmin{i}{\Sigma}$ equal to
$\distmin{i}{\F}$ where $\F$ is a generic foliation $\F \in \Sigma$. Semi-continuity implies that
\[
    \distmin{i}{\Sigma} \ge \distmin{i}{\F}
\]
for every $\F \in \Sigma$.

\subsection{Intersections and spans of pairs of  distributions}
Let $X$ be a projective manifold defined over an algebraically closed field $\field$.
Let $\mathcal D_1$ and $\mathcal D_2$ be two distributions on $X$. We define $\mathcal D_1 \cap \mathcal D_2$
as the distribution on $X$ with tangent sheaf equal to the intersection $T_{\mathcal D_1} \cap T_{\mathcal D_2} \subset T_X$.
If we consider the natural morphism
\begin{align*}
    \varphi : T_{\mathcal D_1} \oplus T_{\mathcal D_2} & \longrightarrow T_X \\
    v_1 \oplus v_2 &\mapsto v_1 + v_2
\end{align*}
then $T_{\mathcal D_1 \cap D_2}$ is isomorphic to the kernel of $\varphi$.

Likewise, we define the span of $\mathcal D_1$ and $\mathcal D_2$, denoted by $\mathcal D_1 + \mathcal D_2$,
as the distribution on $X$ with tangent sheaf equals  the saturation of the image of $\varphi$ in $T_X$.

Our definitions imply that $T_{\mathcal D_1 \cap \mathcal D_2}$ and $T_{\mathcal D_1 + \mathcal D_2}$ fit into
the exact sequence
\[
    0 \to T_{\mathcal D_1 \cap \mathcal D_2} \to T_{\mathcal D_1} \oplus T_{\mathcal D_2} \xrightarrow{\varphi} T_{\mathcal D_1 + \mathcal D_2} \, .
\]
Observe that the right-most arrow has torsion cokernel. Taking determinant, we get the existence of an effective divisor
$\correction{\mathcal D_1}{\mathcal D_2} \in \Div(X)$ with support contained in the support of the cokernel of the right-most arrow such that
\[
    \omega_{\mathcal D_1} \otimes \omega_{\mathcal D_2} = \omega_{\mathcal D_1 \cap \mathcal D_2} \otimes \omega_{\mathcal D_1 + \mathcal D_2} \otimes \mathcal O_X(\correction{\mathcal D_1}{\mathcal D_2}) \, .
\]
When $\codim \mathcal D_1\cap \mathcal D_2 = \codim \mathcal D_1 + \codim \mathcal D_2$ then the divisor $\correction{\mathcal D_1}{\mathcal D_2}$ is equal to the zero divisor of $\omega_1 \wedge \omega_2$, where $\omega_1$ and $\omega_2$ are, respectively, differential forms with coefficients in $\det N_{\mathcal D_1}$ and $\det N_{\mathcal D_2}$ defining $\mathcal D_1$ and $\mathcal D_2$.

Specializing to distributions on projective spaces, \ie  $X = \mathbb P^n_{\field}$, we obtain the identity
\begin{equation}\label{E: degree of intersections}
    \deg(\mathcal D_1 \cap \mathcal D_2) = \deg(\mathcal D_1) + \deg(\mathcal D_2) - \deg(\mathcal D_1 + \mathcal D_2)  - \deg(\correction{\mathcal D_1}{\mathcal D_2}) \, ,
\end{equation}
where we adopt the conventions that the foliation with only one leaf has degree $-1$ and that the foliation by points has degree zero. These conventions are made in order to ensure that $\omega_{\F}$ is isomorphic to $\mathcal O_{\mathbb P^n_{\field}}(\deg(\F) - \dim(\F))$.

\subsection{Bounds for the minimal degrees of subdistributions}
In this subsection, we establish bounds for the integers $\distmin{i}{\F}$. We will make repeated use of the following observation.

\begin{lemma}\label{L:equivalent definition of distmin}
    Let $\field$ be an algebraically closed field and let $\F$ be a codimension one foliation on $\mathbb P^n_{\field}$.
    Then    the vector space
    \[
        H^0(\mathbb P^n_{\field}, \Omega^{i-1}_{\F}(\distmin{i}{\F} - \deg(\F) + (i-1) ))
    \]
    is non zero for any $i$ between $1$ and $n-1$.
\end{lemma}
\begin{proof}
    Let $\omega \in H^0(\mathbb P^n_{\field}, \Omega^1_{\mathbb P^n_{\field}}(\deg(\F)  +2))$ be a twisted $1$-form defining $\F$.
    Let $\mathcal D$ be a distribution on $\mathbb P^n_{\field}$ of degree $\delta$ and codimension $q$. As such, $\mathcal D$ is
    defined by of a twisted $q$-form $\eta \in H^0(\mathbb P^n_{\field}, \Omega^q_{\mathbb P^n_{\field}}(\delta + q +1))$.
    If $\mathcal D$ is everywhere tangent to $\F$ then locally, away from  $\sing(\F)$, we can decompose $\eta$ as the wedge product of $\omega$
    with a $(q-1)$-form $\alpha$. The $(q-1)$-form is not uniquely defined since
    $\omega \wedge \alpha  = \omega \wedge (\alpha  +  \omega \wedge \beta)$ for any $(q-2)$-form $\beta$, but this is the only
    ambiguity in the definition of $\alpha$. Hence  $\restr{\alpha}{T_{\F}}$
    is unique and  we have a global non-zero section of
    $\Omega^{q-1}_{\F} ( \delta + q + 1 - (\deg(\F) + 2))$ on $\mathbb P^n_{\field} - \sing(\F)$. Since $\codim \sing(\F) \ge 2$, we obtain
    a non-zero element of $H^0( \mathbb P^n_{\field}, \Omega^{q-1}_{\F} ( \delta  - \deg(\F) + q - 1))$ as claimed.
\end{proof}

\begin{lemma}\label{L:omegaF em omegaPn}
    Let $\field$ be an algebraically closed field and let $\F$ be a codimension one foliation on $\mathbb P^n_{\field}$.
    For any $i$ between $1$ and $n-1$ we have an injective morphism
    \begin{equation}\label{E:extensao de i-formas}
        \Omega^{i}_{\F} \hookrightarrow \Omega^{i+1}_{\mathbb P^n_{\field}} ( \deg (\F) + 2) \,
    \end{equation}
    with image in $N^*_{\F} \otimes \Omega^i_{\mathbb P^n_{\field}}(\deg (\F) +2)$.
\end{lemma}
\begin{proof}
    Let $\omega \in H^0(\mathbb P^n_{\field}, \Omega^1_{\mathbb P^n_{\field}}(\deg(\F)  +2))$ be twisted $1$-form defining $\F$.
    Since both $\Omega^{i}_{\F}$ and $\Omega^{i+1}_{\mathbb P^n_{\field}} ( \deg (\F) + 2)$ are reflexive sheaves, it suffices
    to define the  morphism on the complement of $\sing(\F)$.
    If $\eta$ is a germ of section of $\restr{\Omega^i_{\F}}{X-\sing(\F)}$ then  we can consider an arbitrary lift of $\eta$ to $\Omega^i_X$ and
    take  the wedge product with $\omega$. Since different lifts differ by $1$-forms proportional to $\omega$, this defines unambiguously
    the sought morphism.
 \end{proof}

\begin{prop}\label{P:obvious sub}
    Let $\F$ be a codimension one foliation on $\mathbb P^n_{\field}$, $n\ge 3$.
    If  the characteristic of $\field$ does not divide $\deg(\F) + 2 = \deg(N_{\mathcal F})$  then $\distmin{i+1}{\F} \le \distmin{i}{\F}$.    Otherwise, $\distmin{i+1}{\F} \le \distmin{i}{\F}+1$.
\end{prop}
\begin{proof}
    Let us first verify that $\distmin{i+1}{\F} \le \distmin{i}{\F}+1$.
    Let  $\mathcal D$ be a codimension $i$ distribution of degree $\distmin{i}{\F}$. To prove the claim, it suffices to
    construct a distribution of codimension $i+1$ and degree $\distmin{i}{\F} + 1$. For that, Equation (\ref{E: degree of intersections}) implies that it suffices to consider the intersection of  $\mathcal D$ with general codimension one foliation $\G$ of degree zero.

    Assume from now on assume that the characteristic of $\field$ does not divide $\deg(N_{\F}) = \deg(\F) + 2$.
    Let $H$ be a hyperplane of $\mathbb P^n$ which is not invariant by $\F$. If $h$ is a linear
    form on $\mathbb A_{\field}^{n+1}$ cutting out $H$ then the logarithmic differential $- \deg(N_{\F})  \frac{dh}{h}$ defines
    a flat logarithmic connection $\nabla$ on $N_{\F}$. Therefore there exists $\eta \in  H^0(\mathbb P^n_{\field}, \Omega^1_{\F} (\log H)) $
    such that
    \[
        \BBconnection - \nabla = \eta \, ,
    \]
    where $\BBconnection : N_{\F} \to N_{\F} \otimes \Omega^1_{\F}$ is Bott's partial connection.
    Since both $\BBconnection$ and $\nabla$ are flat $\F$-connections, $d_{\F} \eta =0$. The non-invariance of $H$ and the hypothesis
    on the characteristic of the base field imply that $\eta \neq 0$.

    Let $\mathcal D$ be a foliation of codimension $i$ and  degree $\distmin{i}{\F}$. As such, $\mathcal D$ is defined by $\alpha \in H^0(\mathbb P^n_{\field}, \Omega^{i}_{\mathbb P^n_{\field}} ( \distmin{i}{\F} + i + 1)$. If $\mathcal D$ is everywhere tangent to $\F$ then we can, unambiguously,  define  the wedge product
    \[
        \alpha \wedge \eta \in H^0(\mathbb P^n_{\field}, \Omega^{i+1}_{\mathbb P^n_{\field}}(\log H) ( \distmin{i}{\F} + i +1))) \, ,
    \]
    see  Lemma \ref{L:omegaF em omegaPn}. Clearing out the denominator (\ie multiplying $\omega\wedge \eta$ by a defining equation of $H$)
    we can interpret $\alpha\wedge \eta$ as a regular section of $\Omega^{i+1}_{\mathbb P^n_{\field}}( \distmin{i}{\F} + i +2)$.
    Moreover, if $H$ is sufficiently general then $\alpha \wedge \eta$ is nonzero and defines a codimension $i+1$ distribution contained in $\F$  and of degree at most $\distmin{i}{\F}$  as wanted.
\end{proof}

\begin{remark}
    The proof of the Proposition \ref{P:obvious sub} can be phrased in more elementary terms. Choose an affine subspace $\mathbb A^{n}_{\field} \subset \mathbb P^n_{\field}$  such that $H$ becomes the hyperplane at infinity. Since $H$ is not invariant, the restriction of $\F$ to $\mathbb A^n_{\field}$  is defined by a $1$-form $\omega$ with coefficients of degree at most $d+1=\deg(\F)+1$.
    If we expand $\omega$ as the sum of its homogeneous components $\omega_0 + \ldots + \omega_{d+1}$ then $\omega_{d+1} \neq 0$ and
    its  contraction with the Euler vector field $R$ is identically zero. Euler's formula implies that
    \[
        i_R d \omega_{d+1} = (d+2) \omega_{d+1}  = \left(\deg(N_{\F})\right) \omega_{d+1} \, .
    \]
    Thus if the characteristic of $\field$ does not divide $\deg(N_{\F})$, $d \omega$ is a non-zero two-form
    defining a codimension two foliation $\G$ tangent to $\F$. The foliation $\G$ coincides with the foliation defined by the section  $\eta \in H^0(\mathbb P^n_{\field}, \Omega^1_{\F} (\log H))$ figuring in the proof presented above. To conclude, Equation \ref{E: degree of intersections} implies that it suffices to  take a codimension $i$ distribution realizing $\distmin{i}{\F}$ and consider the intersection of it with the foliation $\G$.
\end{remark}

\begin{ex}\label{E:contact distribution}
   The assumption on the characteristic of the base field in Proposition \ref{P:obvious sub} is necessary. If $\F$ is a general foliation codimension one foliation of degree zero on $\mathbb P^n_{\field}$, $n\ge 3$, and the characteristic of $\field$ is two then $\distmin{2}{\F}=1 =  \deg(\F)+1$.
\end{ex}

\subsection{Subdistributions of foliations defined by logarithmic $1$-forms}
In this subsection, we will compute the integers $\distmin{2}{\F}$ for a codimension one holomorphic foliation defined by
a logarithmic $1$-form with simple normal crossing divisors showing. In particular, we will show that the bound provided by Proposition \ref{P:obvious sub} is sharp.

\begin{prop}\label{P:Fclosed}
    Let $D$ be a simple normal crossing divisor on $\mathbb P^n_{\C}$.
    Let $\omega \in H^0(\mathbb P^n_{\C}, \Omega^1_{\mathbb P^n_{\C}}(\log D))$ be
    a logarithmic $1$-form with polar divisor equal to $D$. If $\F$ is the foliation defined by $\omega$ and $\eta \in H^0(\mathbb P^n_{\C} , \Omega^i_{\F})$ then $d_{\F} \eta =0$. Moreover, if $j >0$ then $H^0(\mathbb P^n_{\C} , \Omega^i_{\F}(-j))=0$ for every $i$ between $1$ and $\dim \F -1 = n-2$.
\end{prop}
\begin{proof}
    If $\eta \in H^0(\mathbb P^n_{\C}, \Omega^i_{\F})$ is non-zero then we can define the wedge product of $\eta$ with $\omega$, as in  Lemma \ref{L:omegaF em omegaPn},  and obtain a non-zero section of $H^0(\mathbb P^n_{\field}, \Omega^{i+1}_{\mathbb P^n_{\C}}(\log D))$. Since logarithmic forms with poles on simple normal crossing divisors are closed according to Deligne's result \cite[Corollary 3.2.14]{MR498551}, we deduce that $d (\omega\wedge \eta) =0$. It follows that $d_{\F}\eta =0$, showing the first claim of the statement.

    Let now $j$ be an strictly positive integer, let $i$ be a positive integer between $1$ and $n-2$, and let $\eta$ be a non-zero element of
    $H^0(\mathbb P^n_{\C}, \Omega^{i}_{\F}(-j))$. If $f_1, f_2$ are two arbitrary sections of $\mathcal O_{\mathbb P^n_{\C}}(j)$ then
    $f_1 \eta$ and $f_2 \eta$ belong to $H^0(\mathbb P^n_{\C}, \Omega^{i}_{\F})$. Therefore $d_{\F}(f_1 \eta) = d_{\F} (f_2 \eta) =0$ as shown above.
    It follows that
    \[
        d_{\F} \left(\frac{f_1}{f_2}\right) \wedge \eta
    \]
    is identically zero. If $i < \dim \F$ then $f_1/f_2$ is a first integral for
    the subdistribution of $\F$ defined by $\eta$. Since $f_1$ and $f_2$ are arbitrary sections of $\mathcal O_{\mathbb P^n_{\C}}(j)$, we obtain a contradiction proving the second claim.
\end{proof}

\begin{cor}\label{C:distmin for logs}
    Let $n\ge 3$ be a positive integer and Let $D$ be a simple normal crossing divisor on $\mathbb P^n_{\C}$ with $\ell \ge 2$ distinct irreducible components. Let $\omega \in H^0(\mathbb P^n_{\C}, \Omega^1_{\mathbb P^n_{\C}}(\log D))$ be
    a logarithmic $1$-form with polar divisor equal to $D$, and let $\F$ be the codimension one foliation defined by $\omega$.
    Then  $\distmin{2}{\F} = \deg(\F)$ when $\ell =2$, and $\distmin{2}{\F} = \deg(\F)-1$ otherwise.
\end{cor}
\begin{proof}
    Since $D$ is normal crossing and $(\omega)_{\infty} = D$, the $1$-form $\omega$, seen as a section of $\Omega^1_{\mathbb P^n_{\C}}(\log D)$,  has no zeros at a neighborhood of the
    support of $D$. The ampleness of $D$ implies that the zero set of $\omega$ consists of finitely many points, as otherwise positive dimensional components of the zero set of $\omega$ would necessarily intersect $D$. Therefore $N_{\F} = \mathcal O_{\mathbb P^n_{\C}}(D)$ and $\deg(\F) = \distmin{1}{\F} = \deg D -2$.

    If $\distmin{2}{\F}<\deg(\F) -1$ then $h^0(\mathbb P^n_{\C}, \Omega^1_{\F}(-j))\neq 0$ for some $j>0$.  But this contradicts Proposition \ref{P:Fclosed} and, consequently, implies  that $\distmin{2}\F \ge \deg(\F)-1$.

    If $\ell \ge 3$ then the restriction of a general logarithmic $1$-form with poles on $D$ shows that $h^0(\mathbb P^n_{\C}, \Omega^1_{\F})\neq 0$. Hence, when $\ell \ge 3$, $\distmin{2}{\F} = \deg(\F) -1$ as claimed.

    It remains to analyze the case $\ell=2$. We will first show that  $\distmin{2}{\F} \ge \deg(\F)$.
    Aiming at a contradiction, assume  $\distmin{2}{\F} = \deg(\F)-1$. Lemma \ref{L:equivalent definition of distmin} gives a non-zero section $\eta$ of $\Omega^1_{\F}$. The wedge product of $\eta$ with $\omega$, well-defined as explained by Lemma \ref{L:omegaF em omegaPn},   gives a non-zero element of  $H^0(\mathbb P^n_{\C}, \Omega^{2}_{\mathbb P^n_{\C}}(\log D))$.
    Since $D$ has only  $\ell=2$ irreducible components by assumption, the group $H^0(\mathbb P^n_{\C}, \Omega^{2}_{\mathbb P^n_{\C}}(\log D))=0$ is zero, see for instance \cite[Corollary 2.9]{MR4105938}. This gives the sought contradiction showing that $\distmin{2}{\F} \ge \deg(\F)$ when $\ell=2$. Finally, we  apply Proposition \ref{P:obvious sub} to obtain the reverse inequality $\distmin{2}{\F} \le \deg(\F)$.
\end{proof}

\section{Codimension two subdistributions of small  degree}

\subsection{Integrability and uniqueness}
The proof of Corollary \ref{C:distmin for logs} shows that, in general, for a given foliation on a projective space, we do not have uniqueness for the subdistributions realizing $\distmin{i}{\F}$.
Our next result shows that if $\distmin{2}{\F}$ is small enough then we have not only uniqueness but also integrability of the subdistribution realizing $\distmin{2}{\F}$.

\begin{prop}\label{P:integrability and uniqueness}
    Let $\F$ be a codimension one foliation on $\mathbb P^n$, $n \ge 3$, and let $\mathcal D$ be a distribution
    of codimension two contained if $\F$. If $\mathcal D$ is non-integrable then
    \[
        \deg(\mathcal D) \ge  \frac{\deg(\F)+1}{2} \, .
    \]
    Moreover, if there exists another codimension subdistribution $\mathcal D'$ of the same degree of $\mathcal D$
    but different from it then
    \[
        \deg(\mathcal D) \ge  \frac{\deg(\F)}{2} \, .
    \]
\end{prop}
\begin{proof}
    Since $\mathcal D$ is contained in $\F$, the proof of Lemma \ref{L:equivalent definition of distmin} shows that $\mathcal D$
    is defined by a section $\eta$ of $\Omega^1_{\F} ( \deg (\mathcal D) - \deg (\mathcal F) +1)$ . If $\eta$ is non-integrable then
    \[
        0 \neq \theta =  \eta \wedge d_{\F} \eta \in H^0\left(\mathbb P^n, \Omega^3_{\mathcal F}\left( 2(\deg(\mathcal D)- \deg (\F) +1) \right)\right) \,.
    \]
    The $\F$-differential $3$-form $\theta$ determines a non-zero section of
    \[
        \Omega^4_{\mathbb P^n}\left( \deg(\F) + 2 + \left( 2(\deg(\mathcal D)- \deg (\F) +1) \right)\right) \, ,
    \]
    according to  Lemma \ref{L:omegaF em omegaPn}.
    Since $h^0(\Omega^4_{\mathbb P^n}(k)) = 0$ for any $k \le 4$, we deduce that
    $
        \deg(\F) + 2 + \left( 2(\deg(\mathcal D)- \deg (\F) +1) \right) \ge 5 .
    $
    Therefore, the non-integrability of $\mathcal D$ implies
    \[
        \deg(\mathcal D) \ge  \frac{\deg(\F)+1}{2}
    \]
    proving the first claim.

    To proof the second claim, let $\mathcal D'$  be another codimension two distribution
    contained in $\F$. If $\mathcal D$ and $\mathcal D'$ do not coincide then $\mathcal D + \mathcal D' = \F$. Equation (\ref{E: degree of intersections}) implies that
    \[
        \deg(\mathcal D) + \deg(\mathcal D') = \deg(\mathcal D \cap \mathcal D') + \deg(\F) + \deg(\correction{\mathcal D}{\mathcal D'}) \, .
    \]
    Hence, if $\deg(\mathcal D)= \deg(\mathcal D')$ then  $2\deg(\mathcal D) \ge \deg(\F)$   implying the result.
\end{proof}

\subsection{Comparison with the kernel of the $p$-curvature}
We start by settling the notation used throughout this subsection.

Let $\F$ be a codimension one foliation on $\mathbb P^n_{\C}$ be a codimension one foliation.
Let $\mathcal D$ be a codimension $q\ge 2$ distribution, contained in $\F$, and with degree equal to $\distmin{q}{\F}$.
Let us fix, an integral model $(\mathscr F, \mathscr D, \mathbb P^n_{R})$ for  $(\F, \mathcal D, \mathbb P^n_{\C})$ defined
over a finitely generated $\mathbb Z$-algebra $R$. We want to compare $\mathscr D_{\mathfrak p}$ with $\pkernel{\mathscr F_{\mathfrak p}}$
for maximal primes $\mathfrak p$ of $R$.

\begin{lemma}\label{L:comparing}
    If $\mathscr F_{\mathfrak p}$ is not $p$-closed  then the following assertions hold true.
    \begin{enumerate}
        \item The degree of $\pkernel{\mathscr F_{\mathfrak p}}$ is at most $\deg(\F) - 1$.
        \item If $\mathscr D_{\mathfrak p}$ is $p$-closed then $\mathscr D_{\mathfrak p}$ is contained in $\pkernel{\mathscr F_{\mathfrak p}}$.
        \item If $\mathscr D_{\mathfrak p}$ is not contained in $\pkernel{\mathscr F_{\mathfrak p}}$ and then
        \[
            \deg(\pkernel{\mathscr F_{\mathfrak p}}) \ge \deg(\mathscr F_{\mathfrak p}) - \deg(\mathscr D_{\mathfrak p}) + \distmin{q+1}{\mathscr F_{\mathfrak p}} \, .
        \]
    \end{enumerate}
    Moreover, for almost every prime of $R$, $\deg(\mathscr D_{\mathfrak p}) =\distmin{q}{\mathscr F_{\mathfrak p}}$.
\end{lemma}
\begin{proof}
    The first assertion follows from Proposition \ref{P:degree degeneracy}, while the second assertion follows
    from the definition of $\pkernel{\mathscr F_{\mathfrak p}}$.

    Consider the restriction of the $p$-curvature morphism $\pcurvature{\F}$ to
    $\Frob^* T_{\mathscr D_{\mathfrak p}}$. The kernel coincides with $\Frob^* T_{\mathscr D_{\mathfrak p} \cap \pkernel{\mathscr F_{\mathfrak p}}}$. Comparing the cokernels of $\pcurvature{\F}$ and  of   $\restr{\pcurvature{\F}}{\Frob^* T_{\mathscr D_{\mathfrak p}}}$, we deduce that
    \[
         \deg(\mathscr D_{\mathfrak p}) - \deg(\mathscr D_{\mathfrak p} \cap \pkernel{\mathscr F_{\mathfrak p}}) \ge \deg(\mathscr F_{\mathfrak p}) - \deg(\pkernel{\mathscr F_{\mathfrak p}})  \, ,
    \]
    Since $\mathscr D_{\mathfrak p} \cap \pkernel{\mathscr F_{\mathfrak p}}$ has codimension $q+1$,  the inequality  $\deg(\mathscr D_{\mathfrak p} \cap \pkernel{\mathscr F_{\mathfrak p}}) \ge \distmin{q+1}{\mathscr F_{\mathfrak p}}$ holds by definition. The third assertion follows.

    Finally, the fact that $\deg(\mathscr D_{\mathfrak p}) =\distmin{q}{\mathscr F_{\mathfrak p}}$ for almost every prime follows from semi-continuity.
\end{proof}

\subsection{Foliations with codimension two subdistributions of degree zero.}
Foliations on $\mathbb P^n_{\C}$ of degree zero and  codimension $q$ ($0 < q < n$ arbitrary) are easy to describe: they are all defined by the fibers of a linear projection $\mathbb P^n_{\C} \dashrightarrow \mathbb P^q_{\C}$, see for instance \cite[Theorem 3.8]{MR2200857}.

\begin{prop}
    Let $\F$ be a codimension one foliation on $\mathbb P^n_{\mathbb C}$, $n \ge 3$, of degree $d \ge 1$.
    If $\distmin{2}{\F} =0 $ then $\F$ is a linear pull-back of a foliation on $\mathbb P^2_{\C}$.
\end{prop}
\begin{proof}
    According to Proposition \ref{P:integrability and uniqueness} there exists a  subfoliation $\G \subset \F$
    realizing $\distmin{2}{\F}=0$. If $\pi  : \mathbb P^n_{\C} \dashrightarrow \mathbb P^2_{\C}$ is the linear projection
    defining $\G$ then \cite[Lemma 3.1]{MR2324555} implies that $\F = \pi^* \mathcal H$ for a foliation $\mathcal H$ of degree $d$ on
    $\mathbb P^2_{\C}$.
\end{proof}

Let $\Lin{d}{n}{\field} \subset \Fol{d}{n}{\field}$ be the reduced subscheme whose closed points correspond to  foliations on $\mathbb P^n_{\field}$  that are pull-backs of degree $d$ foliations on $\mathbb P^2_{\field}$ under a linear projection $\pi : \mathbb P^n_{\field} \dashrightarrow \mathbb P^2_{\field}$. It is well-known that $\Lin{d}{n}{\mathbb C}$ is an irreducible component of $\Fol{d}{n}{\mathbb C}$ (considered with its reduced structure)
for every $d\ge 0$ and every $n \ge 3$, see for instance \cite[Subsection 5.1]{MR2405162}.

\begin{cor}
    Let $\Sigma \subset \Fol{d}{n}{\C}$ be an irreducible component. If $\distmin{2}{\Sigma}=0$ then
    $\Sigma = \Lin{d}{n}{\C}$.
\end{cor}

\subsection{Foliations with codimension two subdistributions of degree one}\label{SS:cod 2 deg 1}
If $\F$ is a degree one foliation of codimension $q$ on $\mathbb P^n_{\C}$ then  \cite[Theorem 6.2]{MR3066408} gives the following precise description
of $\F$.
\begin{enumerate}
    \item The foliation $\F$ is defined by a dominant rational map $\mathbb P^n_{\C} \dashrightarrow \mathbb P_{\C}(1^q,2)$ with irreducible general fiber determined by $q$ linear forms and one quadratic form; or
    \item The foliation $\F$ is  the linear pullback of a foliation of induced by a global holomorphic vector field on $\mathbb P^{q+1}_{\C}$ and has tangent sheaf isomorphic
    to $\mathcal O_{\mathbb P^n_{\C}}(-1)^{\oplus n-q-1} \oplus \mathcal O_{\mathbb P^n_{\C}}$.
\end{enumerate}

Building on this description, we can prove the following characterization of algebraically integrable foliations of degree one and arbitrary codimension.

\begin{lemma}\label{L:degree one}
    Let $\F$ be a foliation of degree one and codimension $q$ on $\mathbb P^n_{\C}$. The foliation $\F$ is algebraically integrable if, and only if,  $\F$ is $p$-closed for almost every prime.
\end{lemma}
\begin{proof}
    Let $\F$ be a foliation of degree one. If the fibers of a rational map $\mathbb P^n_{\C} \dashrightarrow \mathbb P_{\C}(1^q,2)$ defines $\F$ then there is nothing to prove. Assume from now on that this is not the case. The classification of degree one foliations recalled above implies that $\F$ is the linear pull-back of foliation on $\mathbb P^{q+1}_{\C}$ defined by a global vector field  $v \in H^0(\mathbb P^{q+1}_{\C}, T_{\mathbb P^{q+1}_{\C}})$. The foliation $\F$ is  algebraically integrable if, and only if, the foliation defined by $v$ is algebraically integrable. Represent $v$ by a degree one homogeneous vector field on $\mathbb A^{q+2}_{\C}$ with divergent zero, and let $v = v_S + v_N$ be its Jordan decomposition into semi-simple and nilpotent parts. Explicit integration of $v$ implies that the algebraicity of its orbits is equivalent to
    \begin{enumerate}
        \item the nilpotent part $v_N$ is  zero and the quotient of any  two eigenvalues  of the semi-simple part of $v$ is a rational number; or
        \item the semi-simple part $v_S$ is zero.
    \end{enumerate}
    In the first case, the vector field is tangent to an algebraic action of the multiplicative group $\mathbb C^*$, while in the second case the vector field is tangent
    to an algebraic action of the additive group $\mathbb C$ given by the exponential of $v=v_N$.

    Let $v_R$ be an integral model for $v$ defined over a finitely generated $\mathbb Z$-algebra $R$ contained in $\C$. We can assume that both
    $v_N$ and $v_S$ have integral models over $R$.
    Let $\mathfrak p \subset R$ be a maximal ideal, and set $v_\mathfrak p $ equal to the reduction modulo $\mathfrak p$ of $v_R$.
    Likewise, set $v_{N,\mathfrak p}$ and $v_{S,\mathfrak p}$ as the reduction modulo $\mathfrak p$ of the nilpotent and semi-simple parts
    of $v$.

    Since $[v_S,v_N]=0$, and the same holds for the reduction modulo $\mathfrak p$, Formula (\ref{E:Jacobson}) implies that
    \[
        (v_{S,\mathfrak p} + v_{N,\mathfrak p})^p = v_{S,\mathfrak p}^p + v_{N,\mathfrak p}^p \, .
    \]
    Moreover, if the characteristic of the residue field $R/\mathfrak p$ is sufficiently large (\ie greater than $q+2$)  then $v_{N,\mathfrak p}^p=0$.  Thus, assuming that $p> q+2$, it follows that $v_{\mathfrak p}$ is $p$-closed if, and only, if
    \begin{enumerate}
        \item the nilpotent part $v_{N,\mathfrak p}$ is  zero and the quotient of any  eigenvalues  of the semi-simple part of $v$ belongs to $\mathbb F_p$; or
        \item the semi-simple part $v_{S, \mathfrak p}$ is zero.
    \end{enumerate}
     Combining the above observations with a classical result by Kroenecker (\cite[Theorem 2.2]{MR1975179}) that asserts that an algebraic number is rational if, and only if, its reduction modulo $p$ is in $\mathbb F_p$ for almost every prime $p$, we obtain  that $\F$ is algebraically integrable if, and only if, the reduction of $v_R$ modulo $\mathfrak p$ is $p$-closed for an open set of maximal primes $\mathfrak p$ of $\Spec(R)$, as wanted.
\end{proof}

\begin{prop}\label{P:lembrando deg 3}
    Let $\F$ be a codimension one foliation on $\mathbb P^n_{\mathbb C}$, $n \ge 3$.
    If $\deg(\F) \ge 3$ and $\distmin{2}{\F}= 1$ then one of the following assertions hold true.
    \begin{enumerate}
        \item\label{I:closed rational} The foliation $\F$ is defined by a closed rational $1$-form without codimension one zeros.
        \item\label{I:cod 2 alg int} The foliation $\F$ contains a codimension two algebraically integrable subfoliation $\G$ of degree one.
    \end{enumerate}
\end{prop}
\begin{proof}
    Let $\G \subset \F$ be a codimension two subfoliation of degree one. There is no loss of generality in assuming that $\G$ admits a model $\mathscr G$ over $\Spec(R)$.  Let $\mathfrak p$ be an arbitrary maximal prime of $R$ and let $\mathscr F_{\mathfrak p}$ be the reduction of $\mathscr F_{\mathfrak p}$ of $\mathscr F$ modulo $\mathfrak p$. Reduce $\mathscr G$ modulo $\mathfrak p$ to obtain a subfoliation $\mathscr G_{\mathfrak p}$ of degree one.

    If $\mathscr G_{\mathfrak p}$ is $p$-closed for almost every maximal prime $\mathfrak p$ then Lemma \ref{L:degree one} implies that $\G$ is algebraically integrable as claimed in Item (\ref{I:cod 2 alg int}).

    If $\mathscr G_{\mathfrak p}$ is not $p$-closed for a Zariski dense set of primes then according to the classification of degree one foliations recalled at the beginning of this Subsection we have that $T_{\mathcal G}$ is generated by global sections. We distinguish  two possibilities. Either $\mathscr G_{\mathfrak p}$ coincides with $\pkernel{\mathscr F_{\mathfrak p}}$ or not. Assume first that $\mathscr G_{\mathfrak p}$ coincides with $\pkernel{\mathscr F_{\mathfrak p}}$. In this case, there exists a vector field $v \in H^0(\mathbb P^{n}_{\field} , T_{\mathscr G_{\mathfrak p}})$ such that
    \[
        v^p \in  H^0(\mathbb P^{n}_{\field} , T_{\mathscr F_{\mathfrak p}}) -  H^0(\mathbb P^{n}_{\field} , T_{\mathscr G_{\mathfrak p}}) \, .
    \]
    It follows that we have a morphism
    \begin{align*}
        T_{\mathscr G_{\mathfrak p}} \oplus \mathcal O_{\mathbb P^{n}_{\field}} &\longrightarrow  T_{\mathscr F_{\mathfrak p}} \\
        w \oplus f & \mapsto w + f v^p
    \end{align*}
    generically surjective. Therefore $\deg(\mathscr F_{\mathfrak p}) \le 2$ contrary to our assumptions.

    It remains to treat the case where $\mathscr G_{\mathfrak p}$ does not coincide with $\pkernel{\mathscr F_{\mathfrak p}}$. In this case, there exists a vector field $v \in H^0(\mathbb P^{n}_{\field} , T_{\mathscr G_{\mathfrak p}})$ such that
    \[
        v^p \notin  H^0(\mathbb P^{n}_{\field} , T_{\mathscr F_{\mathfrak p}})  \, .
    \]
    Notice that $v$ must have zero set of codimension at least two, as otherwise $v^p$ would be proportional to $v$. If $\omega_{\mathfrak p} \in H^0(\mathbb P^n_{\field}, \Omega^1_{\mathbb P^n_{\field}} \otimes N_{\mathscr F_{\mathfrak p}})$ is the twisted $1$-form defining $\mathscr F_{\mathfrak p}$ then Proposition \ref{P:old} implies that the rational $1$-form $\omega_{\mathfrak p}/\omega_{\mathfrak p}(v^p)$ is closed. The divisor defined by the vanishing of $\omega_{\mathfrak p}(v^p)$ is  of degree $d+2$ (independent of $p$) and coincides with $\pdegeneracy{\mathscr F_{\mathfrak p}}$. Proposition \ref{P:hinvariant} implies that $\pdegeneracy{\mathscr F_{\mathfrak p}}$ is $\mathscr F_{\mathfrak p}$-invariant if $p> d+2$. We can thus apply Lemma \ref{L:lifting hyp sub} to lift $\omega_{\mathfrak p}(v^p)$ to characteristic zero and deduce that we are in the situation described by Item (\ref{I:closed rational}).
\end{proof}

\begin{lemma}\label{L:closed is closed}
    The set  in $\Fol{d}{n}{\mathbb C}$ corresponding to foliations defined by a closed rational $1$-form without codimension one zeros is closed.
\end{lemma}
\begin{proof}
    A foliation is defined by a closed rational $1$-form without codimension one zeros if, and only if, it admits a
    polynomial integrating factor see \cite[Subsection 3.3]{MR4354288} in particular \cite[Remark 3.4]{MR4354288}.
    The result is a restatement of \cite[Lemma 3.6]{MR4354288}.
\end{proof}

\begin{lemma}\label{L:23}
    Let $\F$ be a  $p$-dense foliation $\F$ on a projecive surface $X$ defined over a field of characteristic $p>0$.
    If $\omega_{\F}$ intersects non-negatively any ample divisor,  $\F$  is defined by a closed rational $1$-form  without codimension one zeroes,
    and the divisor $\pdegeneracy{\F}$ is free from $p$-th powers then $\omega_{\F}^{\otimes p} = \mathcal O_X$.
\end{lemma}
\begin{proof}
    If $\alpha$ is any rational $1$-form defining $\F$ then
    the normal bundle of $\F$ is isomorphic to $\mathcal O_X((\alpha)_{\infty} - (\alpha)_0)$.
    By assumption, there exists a closed rational $1$-form $\eta$, with $(\eta)_0=0$, defining $\F$.
    Therefore $N_{\F} = \mathcal O_X((\eta)_{\infty})$.

    According to Proposition \ref{P:degree degeneracy} the line-bundles
    $\mathcal O_X(\pdegeneracy{\F})$ and $N_{\F}\otimes \omega_{\F}^{\otimes p}$ are isomorphic.
    According to Proposition \ref{P:closed + divisor}, the divisors
    $(\eta)_{\infty}$ and $\pdegeneracy{\F}$ coincide modulo $p$.  Thus we can write
    $(\eta)_{\infty}  = \pdegeneracy{\F} + p D$, for some divisor $D$ such that $\mathcal O_X(pD) = \omega_{\F}^{\otimes p}$.
    Since both $(\eta)_{\infty}$ and $\pdegeneracy{\F}$ are effective
    and $\pdegeneracy{\F}$ is free from $p$-th powers we deduce that $pD = 0$, \ie $\omega_{\F}^{\otimes p}= \mathcal O_X$ as claimed.
\end{proof}

\subsection{Irreducible components with $\distmin{2}{\Sigma}=1$}
Proposition \ref{P:lembrando deg 3} could be easily proved without using reduction to positive characteristic. Instead of considering the $p$-th powers of vector fields tangent to a subfoliation of degree one, one
could consider the Zariski closure of the subgroup generated by the flow of these vector fields as is done in \cite[Subsection 3.4]{MR4354288}.

Our next result seems to be of different nature. It generalizes  to arbitrary degree, an irreducible component of $\Fol{3}{n}{\C}$ found in \cite{MR4354288}. In degree three, the original proof relies on the structure theorem for degree three foliations on projective spaces established in \cite{MR4288634} ($n=3$), and \cite[Theorem A]{MR4354288} ($n \ge 3$). For arbitrary degrees, we are not aware of any proof that does not rely on the reduction to positive characteristic. We have reasons to believe that the standard arguments of the subject (stability of local singular type under deformation as in \cite{MR1838975} or infinitesimal methods as in \cite{MR2405162, MR2590385, MR1838975}) are not sufficient to prove this result. Concerning the stability of singular type, we point out that general foliation described by Theorem \ref{T:P1P1} below has singularities of arbitrary algebraic multiplicity along two skew lines, making it hard to believe that local arguments would show the stability of them. Concerning infinitesimal methods, computer-aided calculations show that the corresponding irreducible component of $\Fol{3}{n}{\C}$ is generically non-reduced. We believe this is also the case
for every degree, as predicted by \cite[Conjecture 6.1]{MR4354288}.

\begin{thm}\label{T:P1P1}
    For every $n \ge 3$ and every positive integers $a\ge 2,b\ge 3$, there exists an irreducible component of $\Fol{a+b-2}{n}{\C}$ whose general elements corresponds to the pull-back of a foliation  $\G$ on $\mathbb P_{\C}^1 \times \mathbb P_{\C}^1$, with normal bundle $N_{\G} = \mathcal O_{\mathbb P^1_{\C} \times \mathbb P^1_{\C}}(a,b)$ under a rational map  $\pi : \mathbb P^n_{\C} \dashrightarrow \mathbb P^1_{\C}\times \mathbb P^1_{\C}$ of the form
    \[
        \pi(x_0: \ldots: x_n)  = (( \ell_1 : \ell_2), (\ell_3: \ell_4)) \, ,
    \]
    where $\ell_1, \ldots, \ell_4 \in H^0(\mathbb P^n_{\C}, \mathcal O_{\mathbb P^n_{\C}}(1))$ are linear forms in general position.
\end{thm}
\begin{proof}
    Let $\G$ be a foliation on $\mathbb P^1_{\C} \times \mathbb P^1_{\C}$ with
    normal bundle $N_{\G} = \mathcal O_{\mathbb P^1_{\C} \times \mathbb P^1_{\C}}(a,b)$.
    Assume that the $1$-form
    $\omega \in H^0( \mathbb P^1_{\C} \times \mathbb P^1_{\C}, \Omega^1_{\mathbb P^1_{\C} \times \mathbb P^1_{\C}}(a,b))$
    is not contained in any proper subvariety of
    $H^0( \mathbb P^1_{\C} \times \mathbb P^1_{\C}, \Omega^1_{\mathbb P^1_{\C} \times \mathbb P^1_{\C}}(a,b))$ defined over $\mathbb Q$.
    If we write down $\omega$ in bihomogeneous coordinates,
    this means that the coefficients of $\omega$ generate a purely transcendental  extension of $\mathbb Q$ with transcendence degree equal to
    $h^0( \mathbb P^1_{\C} \times \mathbb P^1_{\C}, \Omega^1_{\mathbb P^1_{\C} \times \mathbb P^1_{\C}}(a,b))$.

    Let $\Sigma  \subset \Fol{d}{n}{\C}$ be an irreducible component containing $\F= \pi^*\G$. Let $R$ be a finitely generated $\mathbb Z$-algebra
    such that both $\G$ and $\Sigma$ admit integral models defined over $R$.
    Our assumptions on $\omega$ imply that  the reduction of the integral model
    $\mathscr \G$ of $\G$ modulo $\mathfrak p$ is a generic foliation for
    a Zariski dense set of maximal primes $Z \subset \Spec(R)$ and, Theorem \ref{T:Wodson},
    implies that the degeneracy divisor of its $p$-curvature is reduced.

    Since we are assuming that $a\ge 2$ and $b\ge 3$, the line-bundle $\omega_{\G} = \mathcal O_{\mathbb P^1_{\C} \times \mathbb P^1_{\C}}(a-2,b-2)$ is not
    torsion. Therefore, Lemma \ref{L:closed is closed} guarantees that the foliation $\mathscr G_{\mathfrak p}$ is not defined by a closed rational $1$-form without  codimension one zeros for $\mathfrak p \in Z$.

    Proposition \ref{P:pdegeneracy of pb} implies that the degeneracy divisor of the $p$-curvature of
    $\mathscr F_{\mathfrak p}=\pi^*\mathscr G_{\mathfrak p}$ is also reduced for the  Zariski dense set $Z$ of primes.
    For every maximal prime $\mathfrak p$ in $Z$,  the kernel of the $p$-curvature of $\mathscr F_{\mathfrak p}$ is the $p$-closed foliation of degree one defined by
    the reduction modulo $\mathfrak p$ of the rational map $\pi$ .
    We apply Lemma \ref{L:constant pkernel} to deduce that the kernel of the
    $p$-curvature for the reduction modulo $\mathfrak p$ of a  general foliation in $\Sigma$ also has degree one, \ie  $\distmin{2}{\Sigma} = 1$.
    Proposition \ref{P:lembrando deg 3} shows that a general element of $\Sigma$ corresponds to a holomorphic foliation  containing a
    unique algebraically integrable codimension two subfoliation  of degree one.
    It remains to verify that this codimension two subfoliation  is conjugated to the foliation defined by the fibers of $\pi$.

    Let us denote by $0\in \Sigma$ the point corresponding to $\F=\F_0$,
    and by $\mathcal F_{\epsilon}$ the foliation corresponding to a point $\epsilon \in \Sigma$ close to $0$.
    Let also $\mathcal A_{\epsilon}$ be the unique codimension two foliation of $\F_{\epsilon}$ of degree one.
    If $n=3$ then  $\mathcal A_{\epsilon}$ is defined by a global holomorphic vector field $v_{\epsilon}$
    such that $H^0(\mathbb P^3_{\mathbb C}, T_{\F_{\epsilon}})= \mathbb C\cdot v_{\epsilon}$.
    As we did in the proof of Lemma
    \ref{L:degree one}, we can represent $v_{\epsilon}$ by a homogenous vector field of degree one and divergence zero
    on $\mathbb A^4_{\mathbb C}$.
    Since $h^0(\mathbb P^3_{\mathbb C}, T_{\F_{\epsilon}})$ is constant
    for $\epsilon$ in a neighborhood of $0$, we can choose homogeneous representatives of $v_{\epsilon}$
    varying holomorphically with $\epsilon$. The vector field $v_0$ is semi-simple and, after multiplication by suitable
    constant,  has two eigenvalues  equal to $1$ and two other eigenvalues equal to $-1$. The quotient of any two of them is
    an integer. Since the orbits of $v_{\epsilon}$ are algebraic, $v_{\epsilon}$ must remain semi-simple with rational quotient
    of eigenvalues. Since $v_{\epsilon}$ varies holomorphically with $\epsilon$, the quotient of eigenvalues must remain equal
    to the ones of $v_0$. This shows that $\mathcal A_{\epsilon}$ is conjugated to $\mathcal A$ when $n=3$ as claimed. If $n>3$ then the
    restrictions of $\mathcal A_{\epsilon}$ and $\mathcal A_0$ to general $\mathbb P^3_{\mathbb C}$ are conjugated. It follows from
    the classification of degree one foliations, recalled at the
    beginning of Subsection  \ref{SS:cod 2 deg 1}, that the same holds true for $\mathcal A_{\epsilon}$ and $\mathcal A_0$ before taking
    the restrictions. To conclude,  apply \cite[Lemma 3.1]{MR2324555} to guarantee that the general element of $\Sigma$ is pull-back
    from  $\mathbb P^1_{\C} \times \mathbb P^1_{\C}$ through a rational map of the claimed form.
\end{proof}

\begin{thm}\label{T:P112}
    For every $n\ge 3$  and every $d\ge 3$, there exists an irreducible component $\Sigma$ of $\Fol{d}{n}{\C}$ whose general element corresponds to
    the pull-back of a general foliation $\G$ on $\mathbb P(1,1,2)$ with normal sheaf $N_{\G}= \mathcal O_{\mathbb P_{\C}(1,1,2)}(d+2)$ under
    a rational map $\pi : \mathbb P^n_{\C} \dashrightarrow \mathbb P_{\C}(1,1,2)$ defined by two linear forms and one quadratic form.
\end{thm}
\begin{proof}
    Observe that the blow-up of $\mathbb P_{\C}(1,1,2)$ at its unique singular point is isomorphic to a Hirzebruch surface having a section with
    self-intersection $-2$. From Theorem \ref{T:Wodson} we deduce the existence of foliations $\G$ on $\mathbb P_{\C}(1,1,2)$
    with reduced degeneracy divisor of their $p$-curvatures for a Zariski dense set of primes whenever
    $\omega_{\G} = \omega_{\mathbb P_{\C}(1,1,2)} \otimes N_{\G} = \mathcal O_{\mathbb P_{\C}(1,1,2)} ( -4 + d+ 2)$ is nef, \ie,
    whenever $d\ge 2$. If we set $Y$ equal to the smooth locus of $\mathbb P_{\C}(1,1,2)$ and $X$ equal to the pre-image of $Y$ under $\pi$
    then  Corollary \ref{C:pdeg preciso} implies that the $p$-degeneracy of the pull-back of $\G$ under $\restr{\pi}{X}$ is reduced since
    $\Ram(\restr{\pi}{X})=0$. This shows that when $d\ge 2$ we have the existence of $p$-dense foliations $\G$ such that the degeneracy divisor
    of the $p$-curvature of $\pi^*\G$ is reduced for a Zariski dense set of primes.

    Lemma  \ref{L:closed is closed} implies that a general $\G$ on $\mathbb P_{\C}(1,1,2)$
    with $N_{\G}= \mathcal O_{\mathbb P_{\C}(1,1,2)}(d+2)$ is not defined by a closed rational $1$-form without codimension one zeros when $d\ge 3$.

    Arguing as in the proof of Theorem \ref{T:P1P1}, we deduce that sufficiently small deformations of $\pi^* \G$ carry an algebraically integrable
    codimension two  subfoliation of degree one defined by rational maps to $\mathbb P_{\C}(1,1,2)$ given by two linear forms and one quadratic
    form and we conclude using \cite[Lemma 3.1]{MR2324555}.
\end{proof}

Minor variations on the arguments used to prove Theorems \ref{T:P1P1} and \ref{T:P112} give the following result.

\begin{prop}\label{P:Pabc}
    Let $1\le a\le b \le c$ be  positive integers without a common factor and such that $(a,b,c)\neq(1,1,2)$, let $\G$ be a codimension one foliation on $\mathbb P_{\C}(a,b,c)$, let $\ell_1, \ldots, \ell_4 \in H^0(\mathbb P^n_{\C}, \mathcal O_{\mathbb P^n_{\C}}(1))$ be linear forms in general position and let $\pi : \mathbb P^n_{\C} \dashrightarrow \mathbb P_{\C}(a,b,c)$ be the rational map defined as
    \[
        \pi(x_0: \ldots: x_n)  = ( \ell_1 \cdot \ell_4^{a-1} : \ell_2 \cdot \ell_4^{b-1} : \ell_3 \cdot \ell_4^{c-1}) \, .
    \]
    Assume that
    \begin{enumerate}
        \item\label{I:G is p dense} the foliation $\G$ is $p$-dense and  the degeneracy divisor of the $p$-curvature of $\pi^*\G$ is reduced for a Zariski dense set of primes; and
        \item\label{I:G not closed} the foliation $\pi^* \G$ is not defined by a closed rational $1$-form without codimension one zeros.
    \end{enumerate}
    Then the foliation $\F = \pi^* \G$ belongs to a unique irreducible component $\Sigma$ of $\Fol{d}{n}{\C}$, where $d = \deg(\F)$,
    such that the general element of $\Sigma$ admits the same description as $\F$.
\end{prop}

\begin{cor}
    For every $n\ge 3$, $m\ge 3$ and every $d\ge m+1$, there exists an irreducible component $\Sigma$ of $\Fol{d}{n}{\C}$ whose general element corresponds to the pull-back of a general foliation $\G$ on $\mathbb P(1,1,m)$ leaving the rational curve $\{z=0\}$ invariant and  with normal sheaf $N_{\G}= \mathcal O_{\mathbb P_{\C}(1,1,m)}(d+m-2)$.
\end{cor}
\begin{proof}
    The existence of $\G$ satisfying the assumptions of Proposition \ref{P:Pabc} follows from Theorem \ref{T:Wodson} and Remark \ref{R:later use} combined with Corollary \ref{C:pdeg preciso}.
\end{proof}

\section{Foliations without subdistributions of small degree}

\subsection{Foliations with $\distmin{2}{\F}= \deg(\F)$.}
If $\distmin{2}{\F} = \deg{\F}$ (\ie $h^0(\mathbb P^n_{\field}, \Omega^1_{\F})=0$) and the characteristic of $\field$ is positive then $\F$ is $p$-closed as otherwise the kernel of the $p$-curvature would defined a subdistribution of degree strictly smaller than $\deg(\F)$ as implied by Proposition \ref{P:degree degeneracy}. We conjecture that over the complex numbers, $\distmin{2}{\F} = \deg(\F)$ implies that $\F \in \Log{d_1,d_2}{n}{\C}$ for some positive integers $d_1$ and $d_2$.

\begin{conj}\label{Conj:log I}
    Let $\Sigma \subset \Fol{d}{n}{\C}$ be an irreducible component. If $\distmin{2}{\Sigma}$ is equal to $d$ then
    there exists integer $d_1$ and $d_2$ such that $\Sigma = \Log{d_1,d_2}{n}{\C}$.
\end{conj}

We will say that codimension one foliation is virtually transversely additive if, and only if, Bott's partial connection extends to a flat logarithmic connection with finite monodromy.

\begin{prop}\label{P:submin = d}
    Let $\F$ be a codimension one foliation on $\mathbb P^n_{\C}$, $n \ge 3$. If $\F$ leaves invariant an algebraic hypersurface
    then $\F$ is virtually transversely additive or  $\distmin{2}{\F} \le \deg(\F) - 1$.
\end{prop}
\begin{proof}
    The proof is analogous to the proof of Proposition \ref{P:obvious sub}. Let $H$ be an invariant hypersurface and consider the
    flat logarithmic connection $\nabla$ on $\mathcal O_X(H)$ with residue divisor  $- H$ and trivial monodromy. It induces a flat logarithmic
    connection on $N_{\F}$ with residue divisor $-\frac{\deg(N_{\F})}{\deg H} H$, which we still denote by $\nabla$.
    The $\F$-invariance of $H$ implies that the restriction of $\nabla$ to $T_{\F}$ is a holomorphic partial connection on $N_{\F}$.
    If it coincides with Bott's connection then $\F$ is virtually transversely aditive. If instead it does not coincide then the
    difference of $\nabla$ and $\BBconnection$ is a non-zero section of $\Omega^1_{\F}$. It follows that $\distmin{2}{\F} \le \deg(\F) - 1$.
\end{proof}

Our next result provides some evidence, admittedly weak,  toward  Conjecture \ref{Conj:log I}.

\begin{prop}\label{P:Halphen}
    Let $\F$ be a codimension one foliation on $\mathbb P^n_{\C}$, $n \ge 3$. If $\F$ admits two distinct invariant algebraic hypersurfaces
    and $\distmin{2}{\F} = \deg{\F}$ then there exists integers $d_1, d_2$ such that $\F \in \Log{d_1,d_2}{n}{\C}$.
\end{prop}
\begin{proof}
    Let $F$ and $G$ be two distinct hypersurfaces invariant by $\F$. The proof of Proposition \ref{P:submin = d} implies the existence
    of two distinct flat logarithmic connections on the normal bundle of $\F$. Their difference is a non-zero closed logarithmic $1$-form $\eta$ such that $\restr{\eta}{T_{\F}}$ is a holomorphic section of $\Omega^1_{\F}$. Since $\distmin{2}{\F}=\deg(\F)$, we have that $h^0(\mathbb P^n_{\field}, \Omega^1_{\F})=0$. It follows that $\eta$ is a closed logarithmic $1$-form defining $\F$. Since the residues
    $\eta$ are rational, it follows that $\F$ is algebraically integrable.

    Let $f: \mathbb P^n_{\C} \dashrightarrow \mathbb P^1_{\C}$ be a rational first integral for $\F$ with irreducible general fiber. A classical theorem by Halphen, see \cite[Theorem 3.3]{MR3066408}, says that any rational map with irreducible general fiber has, at most, two multiple fibers. Moreover, if $r(\F)$ is the number of fibers of $f$ with non-irreducible support then the proof \cite[Theorem 1.2]{MR3066408} implies that $r(\F) =0$. Hence $f$ can be written as the quotient of powers of two irreducible polynomials. Looking at its logarithmic differential, it becomes clear that $\F \in \Log{d_1,d_2}{n}{\C}$ for some $d_1$ and $d_2$ such that $\deg(\F)= d_1 + d_2 -2$.
\end{proof}

\subsection{Foliations with $\distmin{2}{\F}= \deg(\F) - 1$}
Let $\F$ be a logarithmic foliation on $\mathbb P^n_{\field}$  defined by  $\omega \in H^0(\mathbb P^n_{\field}, \Omega^1_{\mathbb P^n_{\field}}(\log D))$ where $D = \sum_{i=1}^R H_i$ is a simple normal crossing divisor and $H_i$ is a hypersurface of degree $d_i$.
Assume that the zero set of $\omega$ has codimension at least two and that the polar divisor of $\omega$ is equal to $D$. Under these assumptions $\F$ is a foliation of degree $d = \sum_{i=1}^r d_i - 2 $ belonging to $\Log{d_1, \ldots, d_r}{n}{\field}$.
The $\F$-invariance of the support of $D$ implies that the restrictions of sections of $\Omega^1_{\mathbb P^n_{\field}}(\log D)$ to $T_{\F}$ have no poles and, therefore, defines a morphism
\[
    H^0(\mathbb P^n_{\field}, \Omega^1_{\mathbb P^n_{\field}}(\log D)) \longrightarrow H^0(\mathbb P^n_{\field}, \Omega^1_{\F})
\]
with kernel generated by $\omega$. Consequently, $h^0(\mathbb P^n_{\field}, \Omega^1_{\F}) \ge r-2$ and, if $r \ge 3$, we obtain that $\distmin{2}{\F} \le \deg(\F)-1$.

\begin{thm}\label{T:characterize log components}
    Let $\F$ be a codimension one foliation on $\mathbb P^n_{\mathbb C}$, $n \ge 3$. Assume that $\F$ is $p$-dense for a Zariski dense set of primes. If $\distmin{2}{\F}= \deg(\F) -1$ then $\F$ belongs to one of the logarithmic components $\Log{d_1,d_2, \ldots, d_r}{n}{\C}$, $r \ge 3$.
\end{thm}
\begin{proof}
    Let $(\mathscr X, \mathscr F)$     be an integral model for $\F$ defined over a finitely generated $\mathbb Z$-algebra $R$ conatined in $\mathbb C$, \ie $\mathscr X = \mathbb P^n_{R}$ and
    $\mathscr F \otimes_{\mathbb Z} \mathbb C = \F$.
    Let $\mathfrak p \in \Spec(R)$ be a maximal prime  such that the reduction $\mathscr{F}_{\mathfrak p}$  of $\mathscr{F}$
    is not $p$-closed. By semi-continuity, we may assume that $\distmin{2}{\mathscr F_p} = \deg(\F) - 1$. Therefore, the kernel of
    $p$-curvature of $\mathscr F_{\mathfrak p}$ is a subdistribution (indeed subfoliation according Theorem \ref{T:lift}) realizing  $\distmin{2}{\mathscr F_{\mathfrak p}}$. Consequently,
    the divisor $\pdegeneracy{\mathscr F_{\mathfrak p}}$ has degree $\deg(N_{\F})$ according to Proposition \ref{P:degree degeneracy}, and
    $\mathscr F_{\mathfrak p}$ is defined by a closed rational $1$-form with polar divisor equal $\pdegeneracy{\mathscr F_{\mathfrak p}}$
    and without zero divisors. We apply Lemma \ref{L:lifting hyp sub} to deduce that the same holds true for $\F$. To conclude we apply \cite[Proposition 3.7]{MR4354288}.
\end{proof}

\begin{conj}\label{Conj:log II}
    Let $\Sigma \subset \Fol{d}{n}{\C}$ be an irreducible component. If $\distmin{2}{\Sigma} = \deg(\F) -1$ then either
    every foliation $\F \in \Sigma$ is algebraically integrable or there exist integers $d_1, \ldots, d_r$ such that $\Sigma = \Log{d_1, \ldots, d_r}{n}{\C}$.
\end{conj}

\begin{remark}
    For any $d \in \{2,3\}$ and any $n \ge 3$, there are examples of irreducible components $\Sigma$ of $\Fol{d}{n}{\C}$
    with $\distmin{2}{\Sigma} = d -1$  with $\Sigma$ not of the form $\Log{d_1, \ldots, d_r}{n}{\C}$. For $d=2$, there is exactly
    one irreducible component with this property, the so called exceptional component. For $d=3$, there are at least two irreducible components
    with this property, the special logarithmic components $\SLog{2,5}{n}{\C}$ and $\SLog{3,4}{n}{\C}$ which parameterize algebraically integrable foliations defined by logarithmic $1$-forms with codimension one zeros. All the known examples
    with this property satisfy Conjecture \ref{Conj:log II}.
\end{remark}

A confirmation of Conjecture \ref{Conj:ESBT} (by  Ekedahl, Shepherd-Barron, and Taylor)  combined with Proposition \ref{P:submin = d} and  Theorem \ref{T:characterize log components} would confirm both Conjectures \ref{Conj:log I} and \ref{Conj:log II}.

\subsection*{Competing interests} The authors declare none.

\providecommand{\bysame}{\leavevmode\hbox to3em{\hrulefill}\thinspace}
\providecommand{\MR}{\relax\ifhmode\unskip\space\fi MR }
\providecommand{\MRhref}[2]{%
  \href{http://www.ams.org/mathscinet-getitem?mr=#1}{#2}
}
\providecommand{\href}[2]{#2}


\begin{thebibliography}{10}

\bibitem{SGA5}
\emph{{Seminaire de g\'eom\'etrie alg\'ebrique du Bois-Marie 1965-66 SGA 5
  dirige par A. Grothendieck avec la collaboration de I. Bucur, C. Houzel, L.
  Illusie, J.-P. Jouanolou et J. -P. Serre. Cohomologie \(\ell\)-adique et
  fonctions L}}, {Lecture Notes in Mathematics. 589. Berlin-Heidelberg-New
  York: Springer-Verlag. XII, 484 p. DM 37.00; \$ 16.30 (1977).}, 1977.

\bibitem{MR1225842}
\emph{Flips and abundance for algebraic threefolds}, Soci\'{e}t\'{e}
  Math\'{e}matique de France, Paris, 1992, Papers from the Second Summer
  Seminar on Algebraic Geometry held at the University of Utah, Salt Lake City,
  Utah, August 1991, Ast\'{e}risque No. 211 (1992) (1992). \MR{1225842}

\bibitem{https://doi.org/10.48550/arxiv.2102.02799}
R.~Abugaliev, \emph{Characteristic foliation on hypersurfaces with positive
  beauville-bogomolov-fujiki square}, 2021.

\bibitem{Kaledin08}
R.~{Bezrukavnikov} and D.~{Kaledin}, \emph{{Fedosov quantization in positive
  characteristic}}, {J. Am. Math. Soc.} \textbf{21} (2008), no.~2, 409--438
  (English).

\bibitem{BlickleSchwede13}
M.~{Blickle} and K.~{Schwede}, \emph{{\(p^{-1}\)-linear maps in algebra and
  geometry}}, {Commutative algebra. Expository papers dedicated to David
  Eisenbud on the occasion of his 65th birthday}, New York, NY: Springer, 2013,
  pp.~123--205 (English).

\bibitem{MR1863738}
J.-B. Bost, \emph{Algebraic leaves of algebraic foliations over number fields},
  Publ. Math. Inst. Hautes \'{E}tudes Sci. (2001), no.~93, 161--221.
  \MR{1863738}

\bibitem{BrionKumar05}
M.~{Brion} and S.~{Kumar}, \emph{{Frobenius splitting methods in geometry and
  representation theory}}, vol. 231, Boston, MA: Birkh\"auser, 2005 (English).

\bibitem{MR3328860}
M.~Brunella, \emph{Birational geometry of foliations}, IMPA Monographs, vol.~1,
  Springer, Cham, 2015. \MR{3328860}

\bibitem{MR84497}
P.~Cartier, \emph{Une nouvelle op\'{e}ration sur les formes
  diff\'{e}rentielles}, C. R. Acad. Sci. Paris \textbf{244} (1957), 426--428.
  \MR{84497}

\bibitem{MR0106223}
\bysame, \emph{Questions de rationalit\'{e} des diviseurs en
              g\'{e}om\'{e}trie alg\'{e}brique},
   Bull. Soc. Math. France \textbf{86} (1958), 177--251. \MR{106223}.


\bibitem{MR1394970}
D.~Cerveau and A.~Lins~Neto, \emph{Irreducible components of the space of
  holomorphic foliations of degree two in {$\mathbf C{\rm P}(n)$}, {$n\geq
  3$}}, Ann. of Math. (2) \textbf{143} (1996), no.~3, 577--612. \MR{1394970}

\bibitem{MR1838975}
D.~Cerveau, A.~Lins~Neto, and S.~J. Edixhoven, \emph{Pull-back components of
  the space of holomorphic foliations on {${\mathbb C}{\mathbb P}(n)$}, {$n\geq
  3$}}, J. Algebraic Geom. \textbf{10} (2001), no.~4, 695--711. \MR{1838975}

\bibitem{MR2324555}
D.~Cerveau, A.~Lins~Neto, F.~Loray, J.~V. Pereira, and F.~Touzet, \emph{Complex
  codimension one singular foliations and {G}odbillon-{V}ey sequences}, Mosc.
  Math. J. \textbf{7} (2007), no.~1, 21--54, 166. \MR{2324555}

\bibitem{MR1975179}
A.~Chambert-Loir, \emph{Th\'{e}or\`emes d'alg\'{e}bricit\'{e} en
  g\'{e}om\'{e}trie diophantienne (d'apr\`es {J}.-{B}. {B}ost, {Y}.
  {A}ndr\'{e}, {D}. \& {G}. {C}hudnovsky)}, no. 282, 2002, S\'{e}minaire
  Bourbaki, Vol. 2000/2001, pp.~Exp. No. 886, viii, 175--209. \MR{1975179}

\bibitem{MR2405162}
F.~Cukierman and J.~V. Pereira, \emph{Stability of holomorphic foliations with
  split tangent sheaf}, Amer. J. Math. \textbf{130} (2008), no.~2, 413--439.
  \MR{2405162}

\bibitem{MR2590385}
F.~Cukierman, J.~V. Pereira, and I.~Vainsencher, \emph{Stability of foliations
  induced by rational maps}, Ann. Fac. Sci. Toulouse Math. (6) \textbf{18}
  (2009), no.~4, 685--715. \MR{2590385}

\bibitem{MR4354288}
R.~C. da~Costa, R.~Lizarbe, and J.~V. Pereira, \emph{Codimension one foliations
  of degree three on projective spaces}, Bull. Sci. Math. \textbf{174} (2022),
  Paper No. 103092, 39. \MR{4354288}

\bibitem{MR498551}
P.~Deligne, \emph{Th\'{e}orie de {H}odge. {II}}, Inst. Hautes \'{E}tudes Sci.
  Publ. Math. (1971), no.~40, 5--57. \MR{498551}

\bibitem{MR2200857}
J.~D\'{e}serti and D.~Cerveau, \emph{Feuilletages et actions de groupes sur les
  espaces projectifs}, M\'{e}m. Soc. Math. Fr. (N.S.) (2005), no.~103, vi+124
  pp. (2006). \MR{2200857}

\bibitem{MR3695844}
S.~Druel, \emph{On foliations with nef anti-canonical bundle}, Trans. Amer.
  Math. Soc. \textbf{369} (2017), no.~11, 7765--7787. \MR{3695844}

\bibitem{MR4105938}
J.~Gargiulo Acea, \emph{Logarithmic forms and singular projective foliations},
Ann. Inst. Fourier (Grenoble) \textbf{70} (2020), no.~1, 171--203.\MR{4105938}

\bibitem{GiraldoPan10}
L.~Giraldo and A.~J. Pan-Collantes, \emph{On the singular scheme of codimension
  one holomorphic foliations in {$\mathbb P^3$}}, Internat. J. Math.
  \textbf{21} (2010), no.~7, 843--858. \MR{2671527}

\bibitem{MR0316453}
A.~Grothendieck and J.~P. Murre, \emph{The tame fundamental group of a formal
  neighbourhood of a divisor with normal crossings on a scheme}, Lecture Notes
  in Mathematics, Vol. 208, Springer-Verlag, Berlin-New York, 1971.
  \MR{0316453}

\bibitem{MR1611822}
A.~ Grothendieck, \emph{Techniques de construction et th\'{e}or\`emes
  d'existence en g\'{e}om\'{e}trie alg\'{e}brique. {IV}. {L}es sch\'{e}mas de
  {H}ilbert}, S\'{e}minaire {B}ourbaki, {V}ol. 6, Soc. Math. France, Paris,
  1995, pp.~Exp. No. 221, 249--276. \MR{1611822}

\bibitem{hartshorne1977}
R.~{Hartshorne}, \emph{{Algebraic geometry}}, vol.~52, Springer, New York, NY,
  1977 (English).

\bibitem{hartshorne1980stable}
\bysame, \emph{Stable reflexive sheaves}, Math. Ann. \textbf{254} (1980),
  no.~2, 121--176. \MR{597077}

\bibitem{Illusie96}
L.~Illusie, \emph{Frobenius et d\'{e}g\'{e}n\'{e}rescence de {H}odge},
  Introduction \`a la th\'{e}orie de {H}odge, Panor. Synth\`eses, vol.~3, Soc.
  Math. France, Paris, 1996, pp.~113--168. \MR{1409820}

\bibitem{Katz1970}
N.~M. Katz, \emph{Nilpotent connections and the monodromy theorem: Applications
  of a result of turrittin}, Publications mathématiques de l'IHES \textbf{39}
  (1970), 175--232.

\bibitem{MR3057950}
J.~Koll\'{a}r, \emph{Singularities of the minimal model program}, Cambridge
  Tracts in Mathematics, vol. 200, Cambridge University Press, Cambridge, 2013,
  With a collaboration of S\'{a}ndor Kov\'{a}cs. \MR{3057950}

\bibitem{MR3066408}
F.~Loray, J.~V. Pereira, and F.~Touzet, \emph{Foliations with trivial canonical
  bundle on {F}ano 3-folds}, Math. Nachr. \textbf{286} (2013), no.~8-9,
  921--940. \MR{3066408}

\bibitem{MR3842065}
\bysame, \emph{Singular foliations with trivial canonical class}, Invent. Math.
  \textbf{213} (2018), no.~3, 1327--1380. \MR{3842065}

\bibitem{MR4288634}
\bysame, \emph{Deformation of rational curves along foliations}, Ann. Sc. Norm.
  Super. Pisa Cl. Sci. (5) \textbf{21} (2020), 1315--1331. \MR{4288634}

\bibitem{WodsonTese}
W.~Mendson, \emph{Folheações de codimensão um em característica positiva e
  aplicações}, Ph.D. thesis at IMPA (2022), 1--95.

\bibitem{Wodson}
\bysame, \emph{Foliations on smooth algebraic surfaces over positive
  characteristic}, Preprint (2022).

\bibitem{MR927960}
Y.~Miyaoka, \emph{Deformations of a morphism along a foliation and
  applications}, Algebraic geometry, {B}owdoin, 1985 ({B}runswick, {M}aine,
  1985), Proc. Sympos. Pure Math., vol.~46, Amer. Math. Soc., Providence, RI,
  1987, pp.~245--268. \MR{927960}

\bibitem{miyaoka1997geometry}
Y.~Miyaoka and T.~Peternell, \emph{Geometry of higher dimensional algebraic
  varieties}, vol.~26, Springer, 1997.


\bibitem{MR4270046}
A.~Patel, A.~N. Shankar, and J.~P. Whang, \emph{The rank two {$p$}-curvature
  conjecture on generic curves}, Adv. Math. \textbf{386} (2021), Paper No.
  107800, 33. \MR{4270046}

\bibitem{MR1904485}
J.~V. Pereira, \emph{Invariant hypersurfaces for positive characteristic vector
  fields}, J. Pure Appl. Algebra \textbf{171} (2002), no.~2-3, 295--301.
  \MR{1904485}

\bibitem{MR3167130}
J.~V. Pereira and F.~Touzet, \emph{Foliations with vanishing {C}hern classes},
  Bull. Braz. Math. Soc. (N.S.) \textbf{44} (2013), no.~4, 731--754.
  \MR{3167130}

\bibitem{Seshadri1960}
C.~S. {Seshadri}, \emph{{L'op\'eration de Cartier. Applications.}},
  {Vari\'et\'es de Picard. S\'em. C. Chevalley 3 (1958/59), No. 6, 26 p.
  (1960).}, 1960.

\bibitem{MR3827814}
A.~N. Shankar, \emph{The {$p$}-curvature conjecture and monodromy around simple
  closed loops}, Duke Math. J. \textbf{167} (2018), no.~10, 1951--1980.
  \MR{3827814}

\bibitem{MR4036447}
C.~Spicer, \emph{Higher-dimensional foliated {M}ori theory}, Compos. Math.
  \textbf{156} (2020), no.~1, 1--38. \MR{4036447}

\bibitem{stacks-project}
The {Stacks Project Authors}, \emph{\textit{Stacks Project}},
  \url{https://stacks.math.columbia.edu}, 2018.

\end{thebibliography}
\end{document}